\documentclass[11pt,a4paper]{article}
\usepackage[T1]{fontenc}
\usepackage[utf8]{inputenc}
\usepackage[english]{babel}

\usepackage[all]{xy}
\usepackage{amsmath}
\usepackage{amssymb}
\usepackage{amsthm}
\usepackage{enumerate}

\usepackage{graphicx}
\usepackage{psfrag}

\usepackage{geometry}

\newtheorem{theo}{Theorem}[section]
\newtheorem{defini}[theo]{Definition}
\newtheorem{proposi}[theo]{Proposition}
\newtheorem{lemma}[theo]{Lemma}
\newtheorem{coro}[theo]{Corollary}
\newtheorem{rem}[theo]{Remark}

\newcommand{\Aa}{{\mathcal A}}
\newcommand{\Bb}{{\mathcal B}}

\newcommand{\Dd}{{\mathcal D}}
\newcommand{\Ee}{{\mathcal E}}

\newcommand{\Gg}{{\mathcal G}}
\newcommand{\Hh}{{\mathcal H}}

\newcommand{\Kk}{{\mathcal K}}
\newcommand{\Ll}{{\mathcal L}}

\newcommand{\Rr}{{\mathcal R}}

\newcommand{\Tt}{{\mathcal T}}

\newcommand{\Vv}{{\mathcal V}}

\newcommand{\id}{{\mathbf 1}}

\newcommand{\CM}{{\mathbb C}}

\newcommand{\NM}{{\mathbb N}}
\newcommand{\QM}{{\mathbb Q}}

\newcommand{\RM}{{\mathbb R}}

\newcommand{\ZM}{{\mathbb Z}}

\newcommand{\Z}{\ZM}
\newcommand{\R}{\RM}

\newcommand{\TR}{{\rm Tr\,}}                       

\renewcommand{\phi}{\varphi}                  

\renewcommand{\tilde}{\widetilde}

\newcommand{\freq}{\textrm{freq}}		
\newcommand{\vol}{\textrm{vol}}		

\newcommand{\dec}{\delta}

\newcommand{\op}{\text{\rm op}}

\newcommand{\pf}{\lambda_{\text{\rm \tiny PF}}}	

\newcommand{\dtr}{\delta^{tr}}
\newcommand{\dlg}{\delta^{lg}}
\newcommand{\rtr}{\rho_{tr}}
\newcommand{\rlg}{\rho_{lg}}

\newcommand{\Rob}{\Rr}
\newcommand{\HS}{{\mathfrak H}}
\newcommand{\HE}{\mathcal H}  
\newcommand{\res}{\mbox{\rm Res}}
\newcommand{\lp}{{l^*}}
\newcommand{\asym}{\stackrel{t\to 0}{\sim}}
\newcommand{\Ue}{W}
\newcommand{\f}{\mathfrak p}

\title{Spectral triples from stationary Bratteli diagrams\footnote{Work supported by the ANR grant {\em SubTile} no. NT09 564112.}}
\author{ J. Kellendonk$^1$, J. Savinien$^2$\\
{\small $^1$ ICJ, Universit\'e Lyon I, France.}\\
{\small $^2$ LMAM, Universit\'e de Lorraine - Metz, France.}
}
\date{\today}

\begin{document}

\maketitle

\begin{abstract}
We define spectral triples for stationary Bratteli diagrams and study associated Dirichlet forms.
We describe several examples, and emphasize the case of substitution tiling spaces, which are foliated spaces with self-similar Cantor transversals, and leaves homeomorphic to $\RM^d$.
We derive two types of Dirichlet forms for tilings: one of transversal type, and one of longitudinal type whose infinitesimal generator is similar to a Laplacian in $\RM^d$. 
The spectrum of the forms is the set of continuous dynamical eigenfunctions.
\end{abstract}


\tableofcontents





\section{Introduction}
Even though noncommutative geometry \cite{Co94} was invented to
describe (virtual) noncommutative spaces it turned out also to provide
new perspectives on (classical) commutative spaces. In particular
Connes' idea of spectral triples aiming at a spectral description of
geometry has generated new concepts, or shed new light on existing
ones, for topological spaces: dimension spectrum, Seeley type
coefficients, spectral state, or Dirichlet forms are notions which are
derived from the spectral triple and we will talk about them here. Indeed,
we study in this paper certain spectral triples for commutative
algebras which are associated with stationary Bratteli diagrams, that
is, with the space of infinite paths on a finite oriented graph. Such
Bratteli diagrams occur in systems with self-similarity such as the
tiling systems defined by substitutions. 

Our construction follows from earlier onces for metric spaces which go
under the name "direct sum of point pairs" \cite{Christensen} or
"approximating graph" \cite{KS10}, suitably adapted to incorporate the
self-similar symmetry. The construction is therefore more rigid. The
so-called Dirac operator $D$ of the spectral triple will depend on a
parameter $\rho$ 
which is related to the self-similar scaling. 
We observe a new feature which, we believe, ought to be interpreted as
a sign of self-similarity: The zeta function is periodic with purely
imaginary period $\frac{2\pi i}{\log \rho}$. Correspondingly, what
corresponds to the Seeley coefficients (in the case of manifolds) in
the expansion of the trace of the heat-kernel $e^{-tD^2}$ is here
given by functions of $\log t$ which are   $\frac{2\pi }{\log
  \rho}$-periodic. This has consequences for the usual formulae for
tensor products of spectral tiples. If we take the tensor product of
two such triples and compare the spectral states $\Tt_{1,2}$ for the
individual factors with the spectral state $\Tt$ of the tensor
product, then a formula like $\Tt(A_1\otimes A_2) =
\Tt_1(A_1)\Tt_2(A_2)$ will not always hold due to resonance phenomena
of the involved periodicities.  

The heat kernel we were referring to above is the kernel of the
semigroup generated by $D^2$ and hence does not involve the
algebra. But the spectral triple gives in principle rise to other
semigroups whose generators may be defined by Dirichlet forms of the
type $Q(f,g) = \Tt([D,f]^*[D,g])$. Here $f,g$ are represented elements
of the algebra and $\Tt$ a state. We will take for $\Tt$ the spectral
state, a common choice, but not the only possible
one. Pearson-Bellissard \cite{PB09}, for instance, choose the standard
operator trace. There is however a difficulty, namely it is a priori
not clear what is the right domain of definition of $Q$. As we will
conclude from this work, the choice of domain is crucial and needs
additional ingredients. This is why we can only discuss rigorously
Dirichlet forms and Laplacians (their infinitesimal generators) in the
second part of the paper, when we consider our applications. 

Our main application will be to the tiling space of a substitution
tiling. In this case the finite oriented graph defining the spectral
triple is the substitution graph. 
Moreover, the spectral triple is essentially described by the tensor
product of two spectral triples of the above type, one for the
transversal and one for the longitudinal direction. There will be thus 
two parameters, $\rho_{tr}$ and $\rho_{lg}$. 
The additional ingredients, which will allow us to define a domain for
the Dirichlet form $Q$, are the dynamical eigenfunctions of the   
translation action on the tiling space. We have to suppose that these
span the Hilbert space of $L^2$-functions on the tiling space and we
are thus lead, by Solomyak's theorem \cite{Sol07}, to consider Pisot
substitutions. This means that the Perron-Frobenius eigenvalue of the
substitution matrix is $\theta^d$, where $\theta$ is a Pisot number
and $d$ the dimension of the tiling. 
Our main result about the Dirichlet form can then be qualitatively
explained as follows: In order to have non-trivial Dirichlet forms the
parameters $\rho_{tr}$ and $\rho_{lg}$ have to be fixed such that
$\rho_{lg} = \theta^{-1}$ and $\rho_{tr} = |\theta'|$ where $\theta'$
is an algebraic conjugate to $\theta$, distinct from $\theta$, with
maximal modulus. The modulus $ |\theta'|$ is strictly smaller than $1$
by the Pisot-property  and larger than $\theta^{-1}$ (equality holds 
only for quadratic unimodular Pisot numbers). It then follows that
the Laplacian defined by the Dirichlet form can be interpreted as an
elliptic operator on the maximal equicontinuous factor of the
translation action on the tiling space. 


\paragraph{Summary of results}
After a quick introduction to spectral triples we are first concerned
with the properties of their zeta functions in the case that the
expansion of the trace of the heat kernel $e^{-tD^2}$ is not simply an
expansion into powers of $t$ but of the type 



\begin{equation} 
\label{eq-intro-heat}
\TR(e^{-tD^2}) \asym f(-\log t) \; t^\alpha \,,
\end{equation}
with $\Re(\alpha)<0$, and $f:\RM_+ \rightarrow \RM$  a bounded locally integrable function such that $\lim_{s\rightarrow 0^+} s\Ll[f](s)$ exists and is non zero, where $\Ll$ is the Laplace transform. A non-constant $f$ in that expansion has consequences which we did not expect at first.
We are lead in Section~\ref{ssec-specstate} to study classes of operators on $\Bb(\HS)$ which have a compatible behavior.
An operator $A\in \Bb(\HS)$  is {\it weakly regular} if there exists a bounded locally integrable function $f_A:\RM_+ \rightarrow \RM$, for which $\lim_{s\rightarrow 0^+} s\Ll[f_A](s)$ exists and is non zero, such that
\begin{equation} 
\label{eq-intro-wreg}
\TR(e^{-tD^2}A) \asym f_A(-\log t) \; t^\alpha \,,
\end{equation}
where $\alpha$ is the same as in equation~\eqref{eq-intro-heat}.
For such operators, the spectral state does not depend on a choice of a Dixmier trace and is given by:
\[
 \Tt(A) = \lim_{s\rightarrow 0^+} \frac{\Ll[f_A](s)}{\Ll[f](s)} \,,
\]
where $f$ is the same as in equation~\eqref{eq-intro-heat}, see Lemma~\ref{lem-specstate}.
We also define {\it strongly regular} operators, for which one has in particular $f_A = \Tt(A) f$ in equation~\eqref{eq-intro-wreg} (see Lemma~\ref{cor-streg}).
Regular operators have an interesting behavior under tensor product, which we will use in the applications to tilings.
If the spectral triple is a tensor product: \((\Aa,\HS,D)=(\Aa_1\otimes\Aa_2,\HS_1\otimes\HS_2,D_1\otimes \id + \chi \otimes D_2)\), where $\chi$ is a grading on $(\Aa_1,\HS_1,D_1)$, then one has:
\[
\Tt(A_1\otimes A_2) = \lim_{s\rightarrow 0^+} \frac{ \Ll[f_{A_1}f_{A_2}](s) }{ \Ll[f_1f_2](s) }\,,
\]
where $f_i$ is as in \eqref{eq-intro-heat} for $D_i$, and $f_{A_i}$ as in \eqref{eq-intro-wreg} for $A_i$, for each factor $i=1,2$ of the tensor product, see Lemma~\ref{lem-prodstate}.
In general, only if both $A_1$ and $A_2$ are strongly regular for the individual spectral triples
the state will factorize as \(\Tt(A_1\otimes A_2)=\Tt_1(A_1)\Tt_2(A_2)\). Here $\Tt_i$ denotes the spectral state of \((\Aa_i,\HS_i,D_i)\), $i=1,2$ (see Corollary~\ref{cor-prodstrongreg}).
It is easy to build examples for which this equality fails for more general operators: for example $\Tt_1(A_1)=\Tt_2(A_2)=0$ and $\Tt(A_1\otimes A_2)\neq 0$.

\medskip

In Section~\ref{sec-ST-Bratteli} we study spectral triples associated with a stationary Bratteli diagram, that is for the C$^\ast$-algebra of continuous functions on the Cantor set 
of (half-) infinite paths on a finite oriented graph.
These depend on the matrix $A$ encoding the edges between two levels in the diagram (called here a {\em graph matrix} and assumed to be primitive), a parameter $\rho\in(0,1)$ to account for self-similar scaling, and a horizontal structure $\hat\Hh$ (a set of edges linking the edges of the Bratteli diagram). 
We determine the spectral information of such spectral triples.
In Theorem~\ref{thm-ST} we derive the Connes-distance, and show under which conditions it yields the Cantor topology on the path space.
We compute the zeta-function $\zeta(z)=\TR(|D|^{-z})$ and the expansion of the heat-kernel.

\noindent {\bf Theorem} (Theorems~\ref{thm-zeta}, and~\ref{thm-heatkernel}, and Remark~\ref{rem-heatkernel} in the main text.) {\em Consider a spectral triple associated with a stationary Bratteli diagram with graph matrix $A$ and parameter $0<\rho<1$. 
Assume that $A$ 
is diagonalizable with eigenvalues $\lambda_1, \cdots,\lambda_p$.
\begin{itemize}
\item
The zeta-function $\zeta$ extends to a meromorphic function on $\CM$ which is invariant under translation \(z\mapsto z+\frac{2\pi \imath}{\log \rho}\).
It has only simple poles and these are at \(\frac{\log\lambda_j+2\pi \imath k}{-\log \rho}, \: k\in \ZM, j=1,\ldots p\).
In particular, the spectral dimension (abscissa of convergence of $\zeta$) is equal to $s_0=\frac{\log \pf}{\log \rho}$, where $\pf$ is the Perron-Frobenius eigenvalue of $A$.
The residue at the pole \(\frac{\log\lambda_j+2\pi \imath k}{-\log \rho}\) is given by $\frac{C^j_{\hat\Hh}\lambda_j}{-\log\rho}$. 
\item The Seeley expansion of the heat-kernel is given by
\[
 \TR(e^{-t D^2}) = \sum_{j : |\lambda_j| > 1} C^{j}_{\hat\Hh}  \, 
\f_{-2\log \rho,\log \lambda_j}({-\log t}) \; t^{\frac{\log \lambda_j}{2\log\rho}}
\ +C^{j_0}_{\hat\Hh}\frac{-\log t}{-2\log \rho}  + h(t) \,
\]
where $h$ is entire, 
$\f_{r,a}$ is an $r$-periodic smooth function, and $j_0$ is such that $\lambda_{j_0}=1$.
\end{itemize}
}

\smallskip

\noindent The constants $C^{j}_{\hat\Hh}$ are given in (\ref{eq-CH}), they
depend on the choice of horizontal edges $\hat\HE$.
The function $\f_{r,a}$ is explicitly given in equations~\eqref{eq-fnper1} and~\eqref{eq-fnper2}, and its average over a period is \(\bar \f_{r,a}=\frac{1}{r} \Gamma(\frac{a}{r})\).
If $A$ is not diagonalizable then $\zeta$ has poles of higher order and the heat-kernel expansion is more involved (with powers of $\log (t)$ depending on the order of the poles) see Remark~\ref{rem-zeta} and Theorem~\ref{thm-heatkernel}.

\medskip

In Section~\ref{sec-tilings} we apply our findings to substitution tiling spaces $\Omega_\Phi$. 
We consider geometric substitutions of the simplest form, as in \cite{Grunbaum}, which are defined by a decomposition rule followed by a rescaling, i.e.\ each prototile is decomposed into smaller tiles, which, when stretched by a common factor $\theta>1$ (the dilation factor) are congruent to some original tile. The result of the substitution on a tile is called a supertile (and by iteration then an $n$-th  order supertile). 
If one applies only the decomposition rule one obtains smaller tiles, which we call microtiles. 

The substitution induces a hyperbolic action on $\Omega_\Phi$. Our approximating graph will be invariant under this action. But $\Omega_\Phi$ carries a second action, that by translation of the tilings. Although the translation action will not play a direct role in the construction of the spectral triple, it will be crucial in Section~\ref{ssec-DirForm} to define a domain for the Dirichlet forms. 

The approximating graph for $\Omega_\Phi$ is constructed with the help of {\em doubly} infinite paths over the substitution graph. {\em Half} infinite paths describe its canonical transversal. 
We use this structure to construct a spectral triple for $C(\Omega_\Phi)$ essentially as a tensor product of two spectral triples, one obtained from the 
substitution graph and the other from the reversed substitution graph. 
Indeed, the first of the two spectral triples describes the transversal and the second the longitudinal part of $\Omega_\Phi$. Since the graph matrix of the reversed graph is the transpose of the original graph matrix we will have to deal with only one set of eigenvalues $\lambda_1,\cdots,\lambda_p$. It turns out wise, however, to keep two dilation parameters $\rho_{tr}$ and $\rho_{lg}$ as independent parameters, although they will later be related to the dilation factor $\theta$ of the substitution. We obtain: 

\noindent {\bf Theorem} (Theorem~\ref{thm-STOmega} in the main text.) {\em 
The spectral triple for $C(\Omega_\Phi)$ is finitely summable with spectral dimension
\[
s_0  = \frac{d\log\theta}{-\log\rtr} +
\frac{d\log\theta}{-\log\rlg}\,, 
\]
which is the sum of the spectral dimensions of the triples associated with the transversal 
and to the longitudinal part.

The zeta function $\zeta(z)$ has a simple pole at $s_0$ with positive residue.

The spectral measure  is equal to the unique invariant probability measure on $\Omega_\Phi$.
}
\medskip

We discuss in Section~\ref{ssect-Pisot} the particularities of Pisot substitutions. These are substitutions for which the dilation factor $\theta$ is a Pisot number: an algebraic integer greater than $1$ all of whose Galois conjugates have modulus less than $1$. Their dynamical system
$(\Omega_\Phi,\RM^d)$ factors onto an inverse limit of $dJ$-tori, 
its maximal equicontinuous factor $\hat E$, where $J$ is the algebraic
degree of $\theta$. The substitution induces a hyperbolic homeomorphism on that inverse limit which allows us to split the tangent space at each point into a stable and an unstable subspace, $S$ and $U$. The
latter is $d$-dimensional and can be identified with the space in which the tiling lives. $S$ can be split further into eigenspaces of the hyperbolic map, namely $S = S_1+S_2$ where $S_2$ is the direct sum of eigenspaces to the Galois conjugates of $\theta$ which are next to leading in modulus,
that is have maximal modulus among the Galois conjugates which are distinct from $\theta$.
This prepares the ground for the study of  Dirichlet forms and Laplacians on Pisot substitution tiling spaces in Section~\ref{ssec-DirForm}. The main issue is to find a domain for the bilinear form on $C(\Omega_\Phi)$ defined by
\[
 Q(f,g) = \Tt\bigl( [D,\pi(f)]^\ast [D,\pi(g)] \bigr) \,.
\]
$Q$ decomposes into two forms $Q_{tr}$ and $Q_{lg}$, a transversal and a longitudinal one, which turn out to be Dirichlet forms on a suitable core, once the parameters have been fixed to $\rtr=|\theta_2|$ and $\rlg=\theta^{-1}$, where $\theta_2$ is a next to leading Galois conjugate of $\theta$. Our main theorem is the following:

\noindent {\bf Theorem} (Theorem~\ref{thm-DirForm} in the main text.) 
{\em Consider a Pisot substitution tiling of $\RM^d$ with Pisot number $\theta$ of degree $J$.
Suppose that the tiling dynamical system has purely discrete dynamical spectrum.
Let $\theta_2,\cdots \theta_L$ be the subleading conjugates of $\theta$ so that in particular 
$\theta_j=|\theta_2|e^{\imath \alpha_j}$ for $2\le j\le L\le J$.
Assume that for all $j\neq j'$ one has
\(
 \alpha_j - \alpha_{j'} + 2\pi k + 2\pi \frac{\log|\theta_2|}{\log\theta} k' \neq 0\,, \forall k,k' \in \ZM \,.
\)
Then the space of finite linear combinations of dynamical eigenfunctions is a core for $Q$ on which it is closable.
Furthermore, $Q = Q_{tr} + Q_{lg}$, and $Q_{tr/lg}$ has generator $\Delta_{tr/lg}=\sum_{h\in \hat\Hh_{tr/lg}}\Delta_{tr/lg}^h$ given on an eigenfunction $f_\beta$ to eigenvalue $\beta$ by}
\begin{eqnarray*}
\Delta_{lg}^h f_\beta &=&  -c_{lg}(2\pi)^2 \freq(t_{s^2(h)})  \beta(a_h)^2 f_\beta ,\\
\Delta_{tr}^h f_\beta &=&  -c_{tr}(2\pi)^2 \freq(t_{s^2(h)})  \langle \tilde{r_h}^\star,\beta \rangle^2_{\RM^{dJ}} f_\beta .
\end{eqnarray*}
We explain the notation as far as possible without going into details. 
A longitudinal horizontal edge $h\in \hat\Hh_{lg}$ encodes a vector of translation between two tiles in some supertile. Whereas 
a transversal horizontal edge $h\in\hat\Hh_{tr}$ encodes a vector $r_h$ which should be thought of as a return vector between supertiles of a given type sitting in an even larger supertile~--~hence as a large vector~--~a longitudinal horizontal edge $h\in\hat\Hh_{lg}$ stands for a vector of translation $a_h$   
in a microtile between even smaller microtiles~--~so is a small vector. $t_{s^2(h)}$ is the tile (up to similarity) from which the translation encoded by $h$ starts and $\freq(t_{s^2(h)})$ its frequency in the tiling. By definition an eigenvalue is an element $\beta\in{\RM^d}^*$ for which exists a function $f_\beta:\Omega_\Phi\to\CM$ satisfying $f_\beta(\omega+a) = e^{2\pi i \beta(a)} f_\beta(\omega)$ (we write the translation action simply by $(a,\omega) \mapsto \omega+a$) but the geometric construction discussed in Section~\ref{ssect-Pisot} allows us to view $\beta$ also as an element of $S\oplus U$ so that $\beta(a) = \langle \tilde a,\beta\rangle$ ($\langle\cdot,\cdot\rangle$ is a scalar product on $S\oplus U$).  
Here we wrote $\tilde a$ for the vector in $U$ corresponding to $a$ via the identification of $U$ with the space in which the tiling lives. Finally ${}^\star:U\to S_2\subset S$ is the reduced star map. This is Moody's star map followed by a projection onto $S_2$ along $S_1$. The values for the constants $c_{tr}$ and $c_{lg}$ are given in (\ref{eq-ctr}) and (\ref{eq-clg}).

The transverse Laplacian can therefore be seen as a Laplacian on the maximal equicontinuous factor $\hat E$.
The longitudinal Laplacian can be written explicitly on $\Omega_\Phi$, and turns out to be a Laplacian on the leaves, namely it reads $\Delta_{lg} = c_{lg} \nabla_{lg}^\dagger \Kk \nabla_{lg}$, where $\nabla_{lg}$ is the longitudinal gradient, and $\Kk$ a tensor, see equation~\eqref{eq-lgLaplace}.

\section{Preliminaries for spectral triples}
\label{sec-prelim}

A spectral triple  $(\Aa, \HS, D)$ for a unital $C^\ast$-algebra $\Aa$ is given by
a Hilbert space $\HS$ carrying a faithfull representation $\pi$ of $\Aa$ by bounded operators, and
an unbounded self-adjoint operator $D$ on $\HS$ with compact resolvent such that, the set of $a\in \Aa$ for which the commutator $[D,\pi(a)]$ extends to a bounded operator on $\HS$ forms a dense subalgebra $\Aa_0 \subset \Aa$.
The operator $D$ is referred to as the Dirac operator. 
In all examples here it will be assumed to be invertible, with compact inverse.
The spectral triple $(\Aa, \HS, D)$ is termed {\em even} if there exists
a $\ZM / 2$-grading operator $\chi$ on $\HS$ which commutes with $\pi(a)$, $a\in \Aa$, and anticommutes with $D$. 

We will consider here the case of commutative $C^\ast$-algebras, $\Aa=(C(X), \|\cdot \|_\infty)$, of continuous functions over a compact Hausdorff space $X$ with the sup-norm, so we may speak about a spectral triple for the space $X$. The spaces we consider will be far from being manifolds and our main interest lies in the differential structure defined by the spectral triple. More specifically we restrict our attention to the Laplace operator(s) defined by it. The approach to defining Laplace operators via spectral triples has been considered earlier, for fractals by \cite{GI03,GI05,La97} and for ultrametric Cantor sets and tiling spaces in \cite{PB09,JS10,KS10}.

\subsection{Zeta function and heat kernel}
Since the resolvent of $D$ is supposed compact $\TR(|D|^{-s})$ 
can be expressed as a Dirichlet series in terms of the
eigenvalues of $|D|$.\footnote{For simplicity we suppose (as will be the case
in our applications) that $ \ker(D)$ is trivial, otherwise we would
have to work with $\TR_{ \ker(D)^\perp }(|D|^{-s})$ or remove the
kernel of $D$ by adding a finite rank perturbation.}
The 
spectral triple is called {\em finitely summable} if the Dirichlet series is summable for some 
$s\in\RM$ and hence defines a function
\[
\zeta(z) = \TR(|D|^{-z})\,,
\]
on some half plane $\{z\in\CM: \Re(z)>s_0\}$ which is called the {\em
  zeta-function} of the spectral triple. The smallest possible value
for $s_0$ in the above (the {\em abscissa of convergence} of the Dirichlet
series) is called the {\em metric dimension} of the spectral triple.
We call $\zeta$ {\em simple} if $\lim_{s\to s_0^+}(s-s_0)\zeta(s)$
exists. This is for instance the case if $\zeta$ can be
meromorphically extended and then has a simple pole at $s_0$.
We will refer then to the meromorphic extension also simply as the
zeta function of the triple.

Another quantity to look at is the heat kernel $e^{-tD^2}$ of the
square of the Dirac operator. Thanks to the Mellin transform
$$ \Gamma(s)\mu^{-2s} = \int_0^\infty e^{-t\mu^2} t^{s-1} dt,$$
where $\mu>0$ and \(\Gamma(s)=\int_0^{+\infty} e^{-t}t^{s-1}dt\) is the Gamma function, one can relate
the zeta function to the heat kernel as follows:
$$ \Gamma(s)\zeta(2s) = \int_0^\infty \TR(e^{-tD^2}) t^{s-1} dt.$$
This of course makes sense only if $e^{-tD^2}$ is trace class for all
$t>0$, which is anyway a necessary condition for finite
summability. Notice that the trace class condition implies also that 
$s\mapsto \int_\delta^\infty \TR(e^{-tD^2})t^{s-1}dt$ is holomorphic
for all $\delta>0$. 

The above formula is particularily usefull if one knows the asymptotic
expansion of $\TR(e^{-tD^2})$ at $t\to 0$, or only its leading term.\footnote{A function $f:\R^{>0}\to \CM$ is asymptotically equivalent
to  $g:\R^{>0}\to \R$ at $t\to 0$, written $f\asym g$, if  
$|f - g|=o(|g|)$.
$f=O(g)$ means that $\exists M>0,\exists \delta>0,\forall 0<t<\delta:
  |f(t)|\leq M|g(t)|$, and $f=o(g)$ 
  means that 
  $\forall\epsilon>0,\exists\delta>0,\forall 0<t<\delta:
  |f(t)|\leq   \epsilon|g(t)|$.}
It is well known that the form of the asymptotic expansion is related to the singularites of the zeta-function \cite{CoMa08,Iochum}. For instance, an expansion of the form 
 $$\TR(e^{-tD^2}) = \sum_\alpha c_\alpha t^{\alpha} + h(t),$$
with $\Re (\alpha)<0$, $c_\alpha\in \CM$ and $h$ a function which is bounded at $0$ (in particular without logarithmic terms like $\log t$)
implies that the zeta-function has a simple pole at $-2{\alpha}$ with
residue equal to $2c_\alpha/\Gamma(-{\alpha})$ and
is regular at $0$ \cite{CoMa08}. We will see however, that the
situation is quite different in our case where we have to replace
$c_\alpha$ by functions that are periodic in $\log t$.    
Recall the Laplace transform of a function $f$ at $s$ :
\begin{equation}
\label{eq-Laplace}
\Ll[f](s):=\int_0^\infty f(x)e^{-sx}dx \,.
\end{equation}
We assume therefore in the sequel that the
asymptotic behaviour of  the trace of the heat-kernel is given by 
\begin{equation}
\label{eq-asym}
\TR(e^{-tD^2})\stackrel{t\to 0}{\sim} f(-\log t) t^{\alpha} 
\end{equation}
where $\alpha<0$ and $f:\R^{\geq 0}\to \R$ is a bounded, locally 
integrable function for which $\lim_{s\to 0}s\Ll[f](s)$ exists and is
different from $0$.
This is the weakest assumption needed for $\zeta$ to be simple and have non-negative abscissa of convergence, and to be able to compute its residui explicitly, as the following Lemma shows (see also Remark~\ref{rem-freg} for a regular example of such $f$). 
\begin{lemma}\label{lem-heat-zeta} If the trace of the heat-kernel
satisfies (\ref{eq-asym})  
then $\zeta$ is simple, has abscissa of convergence
\[s_0 = -2\alpha \quad\mbox{and}\quad
\frac12\Gamma(\frac{s_0}2)\lim_{s\to s_0^+}(s-s_0)\zeta(s) =
\lim_{s\to 0}s\Ll[f](s).\] 
If, moreover, $\mathcal L[f](s)$ admits a meromorphic extension with
simple pole at $0$ and 
$\TR(e^{-tD^2}) -  f(-\log t) t^{\alpha}=O(t^{\beta})$ 
(with $\beta>\alpha$) then
$\zeta(s)$ has a simple pole at $s_0=-2\alpha$ and hence 
$\frac12\Gamma(\frac{s_0}2)\mbox{\rm Res}(\zeta,s_0)=\mbox{\rm
  Res}(\Ll[f],0) = \lim_{s\to 0}s\Ll[f](s)$. 
\end{lemma}  
\begin{proof} 
We adapt the arguments of \cite{Iochum}.
Let 
$h(t) = \TR(e^{-tD^2}) - f(-\log t) t^{\alpha}$ and $M=\sup_x |f(x)|$.
Then for all $\epsilon>0$ exists $\delta\leq 1$ such that
$|h(t)|\leq \epsilon M t^\alpha$ if $t<\delta$. In particular,
$H_\delta(s) := \int_0^\delta h(t) t^{s-1}dt $ satisfies
$|H_\delta(s)| \leq \epsilon M\frac{\delta^{\alpha+s}}{\alpha+s}$,
provided $\alpha+s>0$. 
Now, again for $\alpha+s>0$
$$ \Gamma(s)\zeta(2s) = \int_{0}^\infty\TR(e^{-tD^2})t^{s-1}dt =
\int_0^\delta f(-\log t) t^{\alpha+s-1}dt
+ H_\delta(s) + g_\epsilon(s)$$
where \(g_\epsilon(s) = \int_{\delta}^\infty\TR(e^{-tD^2})t^{s-1}dt\) is holomorphic in $s$.
This shows that $\zeta(2s)$ is finite for $s>-\alpha$.
Furthermore
\[\lim_{s\to  -\alpha^+}(\alpha+s)\int_0^\delta f(-\log t) t^{\alpha+s-1}dt =
\lim_{s\to 0^+}s \int_{-\infty}^{\log \delta} f(-\tau)e^{\tau s}d\tau =
\lim_{s\to 0^+}s \Ll[f](s)\]
where we have used in the last equation that $\lim_{s\to 0^+}s
\int_{\log \delta}^0 f(-\tau)e^{\tau s}d\tau = 0 $. 
Since $\epsilon>0$ is arbitrary we conclude that 
 $$\lim_{s\to -\alpha^+}(\alpha+s)\Gamma(s)\zeta(2s) = \lim_{s\to 0^+}s \Ll[f](s).$$
Hence $s_0 = -2\alpha$ is the abscissa of convergence. 

Now if $h(t)$ is of order $t^\beta$ we can find
$M>0$ and $\delta>0$ such that $|h(t)t^{-\beta}|\leq M$ if
$0<t<\delta$. If $\beta>\alpha$ the function $t\mapsto t^{s+\beta-1}$
is integrable on $(0,\delta)$ as long as $s$ lies in a sufficiently
small neighbourhood of $-\alpha$. Since $s\mapsto t^{s+\beta-1}$ is
holomorphic for
all $t>0$ we find that
$H_\delta(s) = \int_0^\delta (h(t)t^{-\beta}) t^{s+\beta-1}dt$ is
holomorphic near $s=-\alpha$ which shows that $(\alpha+s)\zeta(2s)$ is
holomorphic there, too. 
Thus $\zeta$ has a simple
pole at $-2\alpha$ and we have the above stated formula for its residue.
\end{proof} 
\begin{rem} 
\label{rem-freg}
If $f(\tau) = e^{ia\tau}$ then $\mathcal L[f](s) = \frac{1}{s-ia}$. Thus 
if $f$ is the restriction of a periodic function of class $C^1$ then
upon using its representation as a Fourier series we see that 
$s\mathcal L[f](s)$ extends to an analytic function around $0$ and 
$$\lim_{s\to 0}s\mathcal L[f](s) = \bar f$$
the mean of $f$.
\end{rem}

\subsection{Spectral state}
\label{ssec-specstate}
Given a bounded operator $A$ on $\HS$ such that $|D|^{-s_0}A$ is in the Dixmier ideal
we consider the expression
\[
 \Tt(A) = \TR_\omega(|D|^{-s_0}A) / \TR_\omega(|D|^{-s_0})
 \]
which depends a priori on the choice of Dixmier trace
$\TR_\omega$. With a little luck, however, 
$\lim_{s\to s_{0}^+} \frac1{\zeta(s)} \TR(|D|^{-s}A)$ exists and then \cite{Co94}
 \[
 \Tt(A) = \lim_{s\to s_{0}^+} \frac1{\zeta(s)} \TR(|D|^{-s}A).
 \]
We provide here a criterion for that. Note that the Mellin transform
allows us to write 
\[  \TR(|D|^{-s}A) = \frac1{\Gamma(\frac s2)}
 \int_0^\infty  \TR(e^{-t D^2} A)t^{\frac s2-1}dt.\] 
We call $A\in\Bb(\HS)$ {\em strongly regular} if there exists a number
$c_A$ such that  
\[
\TR(e^{-tD^2}A)-c_A
\TR(e^{-tD^2})=o\bigl(\TR(e^{-tD^2})\bigr).
\]
If $c_A
\neq 0$ one can thus say that $\TR(e^{-tD^2}A)\asym c_A
\TR(e^{-tD^2})$. In the context in which the heat kernel satisfies
(\ref{eq-asym}) it is useful to consider the notion of {\em weakly regular}
operators  $A\in\Bb(\HS)$. These are operators which satisfy
\begin{equation}
\label{eq-wreg}
\TR(e^{-tD^2}A)\asym f_A(-\log t) t^\alpha
\end{equation}
where $\alpha$ is the same as in (\ref{eq-asym}) and $f_A:\R^{\geq 0}\to \CM$ is a bounded, non-zero, locally integrable function for which $\lim_{s\to 0}s\Ll[f_A](s)$ exists.
Clearly, strongly regular operators are weakly regular and $f_A=c_A f$ in this case, where $f$ is given in equation~\eqref{eq-asym} (one actually has $c_A=\Tt(A)$, see Corollary~\ref{cor-streg}).
\begin{lemma}
\label{lem-specstate} 
Assume that
the trace of the heat-kernel satisfies (\ref{eq-asym}) and that  
$A\in \Bb(\HS)$ is weakly regular, that is, satisfies (\ref{eq-wreg}). Then
 $\lim_{s\to s_{0}^+} \frac1{\zeta(s)} \TR(|D|^{-s}A)$ exists
 and is equal to  
 $$\Tt(A) = \lim_{s\to 0}
 \frac{\Ll[f_A](s)}{\Ll[f](s)}.$$
\end{lemma}
\begin{proof} Under the hypothesis, for all $\epsilon$ we can find a $\delta$ such that, if $s>s_0$, 
\[
\left|\int_0^\delta \big(\TR(e^{-t D^2} A) -
  f_A(-\log t) t^{\alpha}\big)t^{\frac s2-1}dt \right|\leq
\epsilon \int_0^\delta |f_A(-\log t)|t^{\alpha+\frac s2-1}dt\leq \epsilon M_A \frac{\delta^{\alpha+\frac s2}}{\alpha+\frac s2}
\]
where $M_A$ is an upper bound for $|f_A(-\log t)|$. 
Since $\int_\delta^\infty \TR(e^{-t D^2} )t^{\frac s2-1}dt$ and hence also $\int_\delta^\infty \TR(e^{-t D^2} A)t^{\frac s2-1}dt$ are finite for all $\delta>0$ we get ($\alpha = -\frac{s_0}2$)
\[\lim_{s\to s_{0}^+} \frac1{\Gamma(\frac s2)\zeta(s)}\left| \int_0^\infty \TR(e^{-t D^2} A)t^{\frac s2-1}dt 
- \int_0^1 f_A(-\log t)t^{\frac {s-s_0}2-1}dt 
 \right|  \leq \epsilon \tilde M_A.\] 
 Notice that $\int_0^1 f_A(-\log t)t^{s-1}dt = \Ll[f_A](s)$. 
 Since  $\epsilon$ was arbitrary we conclude that
\[ \lim_{s\to s_{0}^+} \frac1{\zeta(s)} \TR(|D|^{-s}A) =
\lim_{s\to s_{0}^+} \frac{  \Ll[f_A](\frac{s-s_0}2)}{\Gamma(\frac s2)\zeta(s)}
= \lim_{s\to 0}
 \frac{\Ll[f_A](s)}{\Ll[f](s)}
. \] 
\end{proof}
\begin{coro}
\label{cor-streg}
If $A\in\Bb(\HS)$ is strongly regular, then  \(\TR(e^{-tD^2}A)\asym \Tt(A) \TR(e^{-tD^2}A)\).
In other words, the functions in equations~\eqref{eq-asym} and~\eqref{eq-wreg} satisfy $f_A=\Tt(A) f$.
\end{coro}
\begin{proof}
If $A$ is strongly regular, then it is also weakly regular with $f_A=c_A f$.
The Laplace transform is linear so \(\Ll[f_A](s)=c_A\Ll[f](s)\), and Lemma~\ref{lem-specstate} then implies $c_A=\Tt(A)$.
\end{proof}
Order the eigenvalues of $|D|$ increasingly (without counting
muliplicity) and let $F_n$ be the $n$-th eigenspace of $|D|$.  
\begin{coro}
\label{cor-cesar}
Let $A\in\Bb(\HS)$ and 
\(\bar A_n=\frac{\TR_{{F}_n} (A|_{{F}_n})}{\dim F_n} \). If the limit
\[
\bar A=\lim_{n\to\infty}\bar A_n
\]
exists then $A$ is strongly regular and $\Tt(A)=\bar A$.
\end{coro}
\begin{proof}
Write \(c_n = e^{-t\mu_n^2} \dim F_n\) where $\mu_n$ is the $n$-th
eigenvalue of $|D|$. One has
\begin{eqnarray*}
 \TR(e^{-tD^2} A ) - \bar A \, \TR(e^{-tD^2} ) 
& = & \sum_{n\ge 1} (\bar A_n -  \bar A) c_n 
\end{eqnarray*}
Now fix an $\epsilon>0$, and choose and integer $N_\epsilon$ such that \(|\bar A_n -\bar A|\le \epsilon\) for all $n\ge N_\epsilon$. Then the series of the r.h.s.\ can be bound by 
\((\sup_n|\bar A_n -\bar A| )\sum_{n< N_\epsilon} c_{n} + \epsilon \sum_{n\ge N_\epsilon}
c_{n}\). Using   
\(\sum_{n\ge N_\epsilon} c_{n}\leq  \TR(e^{-t D^2})\) we  find that for all $\epsilon$ exists $C_\epsilon$ such that
\[ 
\Bigl| \TR(e^{-tD^2} A ) - \bar A \TR(e^{-tD^2}) \Bigr| \leq  C_{\epsilon} +
\epsilon \TR(e^{-t D^2}) . 
\]
Since $\TR(e^{-t D^2})$ tends to $+\infty$ if $t$ tends to $0$ this
shows that $\TR(e^{-t D^2}A)\stackrel{t\to 0}{\sim}f_A(-\log t)t^\alpha$ with $f_A =
\bar{A} f$. 
Applying
Lemma~\ref{lem-specstate} we see that $\Tt(A)=\bar{A}$.
\end{proof}

In the commutative case, when $\Aa=C(X)$ for a compact Hausdorff space $X$, we are particularly concerned with operators of the form $A =\pi(f)$, for $f\in C(X)$ or for any Borel-measurable function $f$ on $X$.
By Riesz representation theorem the functional $f\mapsto \Tt(\pi(f))$ gives a measure on $X$, which we call the {\em spectral measure}.

\subsection{Laplacians}
\label{ssec-Laplace}
It is tempting to define a quadratic form by 
\begin{equation}
\label{eq-Dirform}
 Q(a,b) = \Tt([D,\pi(a)]^*[D,\pi(b)])
\end{equation}
on elements on which this expression makes sense.
It is however not so easy to determine a domain for such a form.
Our interest here lies in the commutative case, $\Aa=C(X)$, and our question is as follows:
Let $\mu$ be the spectral measure on $X$ and consider the Hilbert
space $L^2(X) = L^2(X,\mu)$. Notice that this is usually not the Hilbert space of
the spectral triple. We embed $\Aa_0$ (continuous functions having bounded
commutator with $D$) into $L^2(X)$, assuming that the measure is faithful.  
Does there exist a core (a dense linear subspace of $L^2(X)$) which is
contained in $\Aa_0$ and on which $Q$ is well-defined, yielding a
symmetric closable quadratic form which satisfies the Markov
property? If $\Dd$ is the domain of the closure of $Q$ then the
general theory guaranties the existence of a positive operator
$\Delta$ such that $Q(f,g) = (f,\Delta g)$ and this operator generates
a Markov process on $L^2(X)$.
We refer to $\Delta$ as a Laplacian.
We emphasize that $\Delta$ will depend on the choice of a domain.
We won't be able to give general answers here, but we look at specific
examples related to self-similarity.

\subsection{Direct sums}
\label{ssec-sums}
The direct sum of two spectral triples $(\Aa_1,\HS_1,D_1)$ and
$(\Aa_2,\HS_2,D_2)$ is the spectral triple $(\Aa,\HS,D)$  given by 
\[
\Aa = \Aa_1\oplus \Aa_2,\quad \HS = \HS_1\oplus\HS_2,\quad D =
D_1\oplus D_2.
\]
with direct sum representation. The zeta function $\zeta$ of the direct sum is clearly the sum of the zeta functions $\zeta_i$ of the summands and thus its abscissa of convergence $s_0$ is equal to the largest abscissa of the two zeta functions $\zeta_i$.
Let $\Tt$ denote the spectral state of the direct sum triple and $\Tt_i$ those of the summands and assume that all zeta-functions are simple.
Then, for regular operators $A_1$ and $A_2$ we have
\begin{equation}
\label{eq-sum-state}
\Tt(A) = \frac1{c_1+c_2} \bigl( c_1\Tt(A_1) + c_2\Tt(A_2) \bigr)
\end{equation}
where $c_i = \lim_{s\to s_0^+} (s-s_0) \zeta_i(s)$.
Notice that $c_1=0$ if the abscissa of convergence of $\zeta_1$ is smaller than that of $\zeta_2$, in which case $\Tt(A)=\Tt(A_2)$
(and similiarily with $1$ and $2$ exchanged: $\Tt(A)=\Tt(A_1)$).
Notice also that $c_1+c_2 = \lim_{s\to s_0^+} (s-s_0) \zeta(s)$.

It is pretty clear that the quadratic form $Q$ used to define the Laplacian 
is the sum of the quadratic forms of the individual triples, again leaving questions about its domain
aside.

\subsection{Tensor products}
\label{ssec-tensorprod}
The tensor product of an even spectral triple $(\Aa_1,\HS_1,D_1)$ with
grading operator $\chi$ and another spectral triple
$(\Aa_2,\HS_2,D_2)$ is the spectral triple $(\Aa,\HS,D)$  given by 
\[
\Aa = \Aa_1\otimes \Aa_2,\quad \HS = \HS_1\otimes\HS_2,\quad D =
D_1\otimes 1 + \chi\otimes D_2.
\]
Notice that $D^2 = D_1^2\otimes 1 + 1\otimes D_2^2$. It follows that the
trace of the heat kernel is multiplicative: $\TR(e^{-tD^2}) = \TR_{\HS_1}(e^{-tD_1^2})\TR_{\HS_2}(e^{-tD_2^2})$. This allows one to obtain information on the zeta function, the spectral state, and the  quadratic form $Q$ of the tensor product.
\begin{lemma}
\label{lem-prod}
Suppose that $\TR_{\HS_1}(e^{-tD_1^2})$ and $\TR_{\HS_2}(e^{-tD_2^2})$ satisfy (\ref{eq-asym}) with $f = f_1$ and $f=f_2$, respectively. 
Suppose that $\lim_{s\to 0} s\mathcal L[f_1f_2](s)$ exists and is non-zero. Then
the metric dimension $s_0$ of the tensor product
spectral triple is the sum of the metric dimensions of the
factors, and 
the zeta function $\zeta$ of the tensor product is simple with 
\[\frac12\Gamma(\frac{s_0}2)\lim_{s\to s_0^+} (s-s_0)\zeta(s) = \lim_{s\to 0}
s\mathcal L[f_1f_2](s).\]
\end{lemma}
\begin{proof} Due to the multiplicativity of
  the trace of the heat kernel we have $$ \TR(e^{-tD^2})
  \stackrel{t\to 0}{\sim} f_1(-\log t)f_2(-\log t)
  t^{\alpha_1+\alpha_2}$$ 
and hence the result follows from Lemma~\ref{lem-heat-zeta}.
\end{proof}
\begin{lemma}
\label{lem-prodstate} Assume the conditions of Lemma~\ref{lem-prod}.
Let $A_1\in\Bb(\HS_1)$ and $A_2\in\Bb(\HS_2)$ be weakly regular with
functions $f_{A_1}$ and $f_{A_2}$.
Then 
$$\Tt(A_1\otimes A_2) = \lim_{s\to
  0^+}\frac{\Ll[f_{A_1}f_{A_2}](s)}{\Ll[f_1f_2](s)}.$$
\end{lemma}
\begin{proof}
Let $\epsilon>0$ and choose $\delta$ such that $|\TR_{\HS_i}(e^{-tD_i^2}A_i) - f_{A_i}(-\log t)t^{\alpha_i} |\leq \epsilon t^{\alpha_i}$ provided $0<t<\delta$. Then
\begin{eqnarray*} 
\int_0^\delta \TR_\HS(e^{-tD^2}A_1\otimes A_2)t^{s-1}dt &=& 
 \int_0^\delta \TR_{\HS_1}(e^{-tD^2_1}A_1)\TR_{\HS_2}(e^{-tD^2_2} A_2) t^{s-1}dt \\
& =& \int_0^\delta\big(f_{A_1}(-\log t)f_{A_2}(-\log t) +
O(\epsilon)\big) t^{\alpha_1+\alpha_2+s-1}dt \\
& =& \Ll[f_{A_1}f_{A_2}](\alpha_1+\alpha_2+s) + 
O(\epsilon)\frac{\delta^{\alpha_1+\alpha_2+s}}{\alpha_1+\alpha_2+s} \end{eqnarray*}
from which the result follows by arguments similar to the above.
\end{proof}
\begin{coro}
\label{cor-prodstrongreg}
Let $A_1$ and $A_2$ be weakly regular operators.
\begin{enumerate}[(i)]
 \item If $A_1$ is strongly regular then $\Tt(A_1\otimes A_2) = \Tt_1(A_1)\Tt(\id \otimes A_2) $.

 \item If both $A_1$ and $A_2$ are strongly regular then $\Tt(A_1\otimes A_2) = \Tt_1(A_1)\Tt_2(A_2) $.
\end{enumerate}
\end{coro}
\begin{rem}
The result of the corollary says that the spectral state factorizes for tensor products of strongly regular operators.
This corresponds to the formula on page~563 in \cite{Co94}.
It should be noticed, however, that this factorisation is in general not valid for tensor products of weakly regular operators, since the Laplace transform of a product is not the product of the Laplace transforms. We consider below examples of this type.
\end{rem}

\begin{lemma}
\label{lem-prodform}
Let $a_i,b_i\in\Aa_i$, and set \( \{d;a,b\} = [d,\pi(a)]^*[d,\pi(b)]\).
\begin{enumerate}[(i)]
\item Then one has
\[
Q\bigl( a_1\otimes a_2, b_1\otimes b_2 \bigr) = 
\Tt \bigl(  \{ D_1; a_1,b_1 \} \otimes \pi(a_2^\ast b_2) \bigr)
+ \Tt\bigl(  \pi(a_1^\ast b_1) \otimes \{ D_2; a_2,b_2 \} \bigr) \,.
\]
\item If \(\pi(a_i^\ast b_i)\), $i=1,2$, are strongly regular then
\[
Q\bigl( a_1\otimes a_2, b_1\otimes b_2 \bigr) = 
\Tt \bigl(  \{ D_1; a_1,b_1 \} \otimes \id  \bigr) \Tt_2 \bigl( \pi(a_2^\ast b_2) \bigr)
+\Tt_1\bigl( \pi(a_1^\ast b_1)\bigr) \Tt\bigl(  \id \otimes \{ D_2; a_2,b_2 \} \bigr) \,,
\]
\end{enumerate}
\end{lemma}
\begin{proof}
Since $[D,\pi(a_1\otimes a_2)] = [D_1,\pi(a_1)]\otimes \pi(a_2) + \chi \pi(a_1)\otimes [D_2,\pi(a_2)]$ we get 
\begin{eqnarray*}
[D,\pi(a_1\otimes a_2)]^*[D,\pi(b_1\otimes b_2)]&=& [D_1,\pi(a_1)]^*[D_1,\pi(b_1)]\otimes \pi(a_2^*b_2) \\
&&+\pi(a_1^*b_1) \otimes [D_2,\pi(a_2)]^*[D_2,\pi(b_2)]\\
&&+[D_1,\pi(a_1)]^*\chi\pi(b_1) \otimes \pi(a_2)^*[D_2,\pi(b_2)]\\
&&+\chi\pi(a_1^*)[D_1,\pi(b_1)] \otimes [D_2,\pi(a_2)]^*\pi(b_2)
\end{eqnarray*}
Since $\TR(\chi A) = 0$ for any odd operator $A\in \Bb(\HS_1)$ the contributions of the last two lines vanish under the trace, hence under $\Tt$, and we get the first statement.
Point (ii) follows from Corollary~\ref{cor-prodstrongreg}.
P
\end{proof}

\section{The spectral triple associated with a stationary Bratteli diagram}
\label{sec-ST-Bratteli} 

An oriented graph \(\Gg=(\Vv, \Ee)\) is the data of a set of vertices $\Vv$ and a set of edges $\Ee$ with two maps,
\(\xymatrix{\Ee \ar@<.3ex>[r]^{r} \ar@<-.3ex>[r]_{s}  & \Vv}\), one assigning to an edge $\varepsilon$ its source vertex $s(\varepsilon)$ and the second assigning its range $r(\varepsilon)$.
The {\em graph matrix} of $\Gg$ is the matrix $A$ with coefficients $A_{vw}$ equal to the number of edges which have source $v$ and range $w$.

We construct a spectral triple from the following data (see Figure~\ref{fig-graph} for an illustration of the construction):
\begin{enumerate}
\item A finite oriented graph $\Gg=(\Vv,\Ee)$ with a distinguished one-edge-loop $\lp$.
We suppose that the graph is strongly connected :  for any two vertices $v_1,v_2$ there exists an oriented path from $v_1$ to $v_2$ and an oriented path from $v_2$ to $v_1$. This is equivalent to saying that the graph matrix $A$ is irreducible.
We will further assume that $A$ is {\em primitive} (see below). 

Alternatively we can pick a distinguished loop made of $p>1$ edges, and resume the case described above by replacing the matrix $A$ by $A^p$ and considering its associated graph $\Gg^p$ instead of $\Gg$.

\item A function $\hat\tau:\Ee\to\Ee$ satisfying: for all $\varepsilon\in \Ee$
\begin{enumerate}
\item if $r(\varepsilon)$ is the vertex of $\lp$ then $\hat\tau(\varepsilon)=\lp$,
\item if $r(\varepsilon)$ is not the vertex of  $\lp$ then $\hat\tau(\varepsilon)$ is an edge starting at $r(\varepsilon)$ and such that $r(\hat\tau(\varepsilon))$
is closer\footnote{for the combinatorial graph metric, where non-loop edges have length $1$, and loop edges length $0$.} to the vertex of $\lp$ in $\Gg$. 
\end{enumerate}
\item A symmetric subset $\hat \HE=\Hh(\Gg)$ of $\Ee\times\Ee$:
\[
\hat\HE \subseteq \left\{ (\varepsilon,\varepsilon') \in \Ee\times\Ee \, : \, \varepsilon\neq \varepsilon', \; s(\varepsilon)=s(\varepsilon') \right\}
\]
This can be understood as a graph with vertices $\Ee$ and edges $\hat\HE$, which has no loops. 
We fix an orientation of the edges in $\hat\HE$, and write \(\hat\HE = \hat\HE^{+} \cup \hat\HE^{-} \) for the decomposition into positively and negatively oriented edges.

\item A real number $\rho \in (0,1)$.
\end{enumerate}

{\em Notation.} We still denote the range and source maps by $r,s$ on $\hat\Hh$: \(\xymatrix{\hat\Hh \ar@<.3ex>[r]^{r} \ar@<-.3ex>[r]_{s}  & \Ee}\).
We allow compositions with the source and range maps from $\Ee$ to $\Vv$, which we denote by \(s^2, r^2, sr, rs\): \(\xymatrix{\hat\Hh \ar@<.6ex>[r]^{r^2,s^2} \ar@<.3ex>[r] \ar[r] \ar@<-.3ex>[r]_{rs,sr}  & \Vv}\).
See Figure~\ref{fig-graph} for an illustration with the Fibonacci matrix $A=\left(\begin{array}{cc} 1 & 1 \\ 1 & 0 \end{array}\right)$.
\begin{figure}[htp]
\begin{center}
\psfrag{a}{$s(\varepsilon)=sr(h)$}
\psfrag{b}{$r(\varepsilon)=r^2(h)$}
\psfrag{e}{$\hat{\tau}(\varepsilon)$}
\psfrag{f}{\qquad $\varepsilon=r(h)$}
\psfrag{g}{$\lp=\hat{\tau}\circ \hat{\tau}(\varepsilon)$}
\psfrag{h}{$h\in \hat\Hh^{+}$}
\includegraphics[scale=0.4]{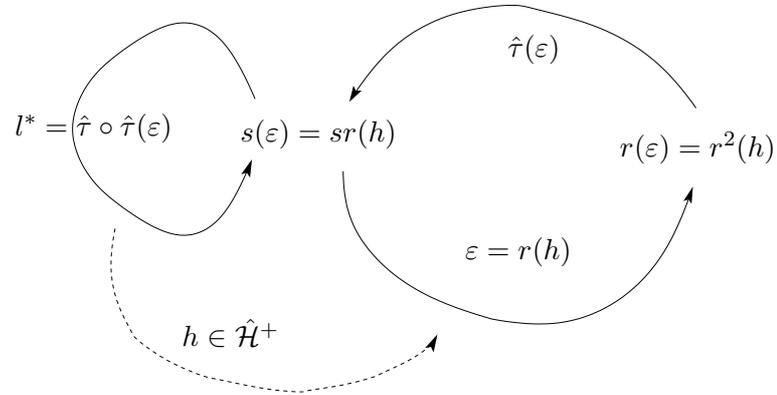}
\caption{{\small The graph $\Gg$ associated with the Fibonacci matrix and horizontal edges $\hat\Hh$.}} 
\label{fig-graph}
\end{center}
\end{figure}
\begin{figure}[htp]
\begin{center}
\[
\xymatrix{
&  \bullet \ar@{->}[ddrr] && \bullet \ar@{->}[ddrr] && \bullet \ar@{.}[dr] &\\
\circ \ar@{->}[ur] \ar@{->}[dr] && && && && \\
& \bullet \ar@{->}[uurr] \ar@{->}[rr] && \bullet \ar@{->}[uurr] \ar@{->}[rr]&& \bullet \ar@{.}[ur] \ar@{.}[r]& 
}
\]
\caption{{\small The stationnary Bratteli diagram associated with the Fibonacci matrix
.}} 
\label{fig-brat}
\end{center}
\end{figure}
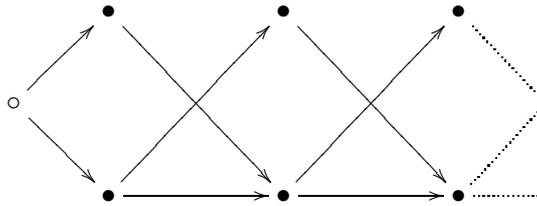

We denote by $\Pi_n(\Gg)$, or simply by $\Pi_n$ if the graph is understood, the set of paths of length $n$ over $\Gg$: sequences of $n$ edges $\varepsilon_1\cdots \varepsilon_n$ such that $r(\varepsilon_i)=s(\varepsilon_{i+1})$.
We also set $\Pi_0(\Gg)=\Vv$.
We extend the range and source maps to paths: if $\gamma=\varepsilon_1\cdots \varepsilon_n$ then $r(\gamma):=r(\varepsilon_n), s(\gamma):=s(\varepsilon_1)$.
Also, given $\gamma= \varepsilon_1 \cdots \varepsilon_i \cdots \varepsilon_n$, we denote by $\gamma_i=\varepsilon_i$ the $i$-th edge along the path.

The number of paths of length $n$ starting from $v$ and ending in $w$ is then ${A^n}_{vw}$.
We require that $A$ is {\em primitive}: \(\exists N \in \NM, \, \forall v,w, \, A^N_{vw}>0\).
(For the graph $\Gg$, this means that for any two vertices $v,w$, there is at least one oriented paths of length $N$ from $v$ to $w$; for the graph $\Gg^N$ this means that for any two vertice $v,w$, there is at least one oriented edge from $v$ to $w$.)
Under this assumption, $A$ has a non-degenerate positive eigenvalue $\pf$ which is strictly larger than the modulus of any other eigenvalue.
This is the Perron-Frobenius eigenvalue of $A$. 
Let us denote by $L$ and $R$ the left and right Perron-Frobenius eigenvectors of $A$ ({\it i.e.}\ associated with $\pf$), normalized so that 
\begin{equation}
\label{eq-normalization}
\sum_{j}R_j = 1,\quad \sum_{j} R_j L_j = 1\, .
\end{equation}
Let us also write the minimal polynomial of $A$ as \(\mu_A(\lambda)=\Pi_{k=1}^p (\lambda-\lambda_k)^{m_k}\) with $\lambda_1=\pf, m_1=1$.
Then from the Jordan decomposition of $A$ one can compute the asymptotics of the powers of $A$ as follows  \cite{HJ94}:
\begin{equation}
\label{eq-An}
A^n_{ij}= R_i L_j \; \pf^n + \sum_{k=2}^p P^{(ij)}_k(n) \;  \lambda_k^n \,,
\end{equation}
where $P^{(ij)}_k$ is a polynomial of degree $m_k$ if $n\ge m_k$, and of degree less than $m_k$ if $n<m_k$. 

Let $M_j$ be the algebraic multiplicity of the $j$-th eigenvalue of $A$ (hence $M_1=1$).
Let $R^{j,l}$, for $1\le j \le p$ and $1\le l \le M_j$, be a basis of (right) eigenvectors of $A$: \(AR^{j,l} = \lambda_j R^{j,l}\).
Let also $L^{j,l}$ be a basis of left eigenvectors of $A$ normalized so that \(R^{j,l} \cdot L^{j',l'}=\delta_{jj'} \delta_{ll'}\).
So $R^{1,1}=R$ and $L^{1,1}=L$ as defined in equation~\eqref{eq-An}.
For $1\le j \le p$, let us define
\begin{equation}
\label{eq-CH}
C^j_{\hat\Hh} = \frac{1}{\lambda_j} \sum_{l=1}^{M_j} 
\sum_{\substack{v\in\Vv \\ h\in\hat\Hh}} R^{j,l}_v L^{j,l}_{s^2(h)}\,.
\end{equation}

Given $\Gg$ we consider the topological space $\Pi_\infty$ of all (one-sided) infinite paths over $\Gg$ with the standard topology. It is compact and metrisable. The set $\Pi_{\infty \ast}$ of infinite paths which eventually become $\lp$ forms a dense set. 

\begin{rem}
\label{rem-brat1}{\em
This construction is equivalent to that of a {\em stationary Bratteli diagram} \cite{Bra72}: this is an infinite directed graph, with a copy of the vertices $\Vv$ at each level $n\ge 0$, and a copy of the edges $\Ee$ linking the vertices at level $n$ to those at level $n+1$ for all $n$ (there is also a root, and edges linking it down to the vertices at level $0$).
So for instance, the set $\Pi_n$ of paths of length $n$, is viewed here as the set of paths from level $0$ down to level $n$ in the diagram.
See Figure~\ref{fig-brat} for an illustration (the root is represented by the hollow circle to the left).
}
\end{rem}

Given $\hat\tau$ we obtain an embedding of $\Pi_n$ into $\Pi_{n+1}$ by $\varepsilon_1\cdots \varepsilon_n\mapsto \varepsilon_1\cdots \varepsilon_n\hat\tau(\varepsilon_n)$ and hence, by iteration, into $\Pi_{\infty \ast}\subset \Pi_\infty$. 
We denote the corresponding inclusion $\Pi_n\hookrightarrow \Pi_\infty$ by $\tau$.

Given $\hat\HE$ we define horizontal edges $\Hh_n$ between paths of $\Pi_n$, 
namely  $(\gamma,\gamma')\in \Hh_n$ if $\gamma$ and $\gamma'$ differ only on their last edges $\varepsilon$ and $\varepsilon'$, and with $(\varepsilon,\varepsilon')\in \hat\HE$.
For all $n$, we carry the orientation of $\hat\HE$ over to $\Hh_n$.

The {\em approximation graph} \(G_\tau=(V,E)\) is given by:
\begin{eqnarray*}
V = \bigcup_{n\ge 0} V_n\,, & V_n = \tau (\Pi_n) \subset \Pi_{\infty\ast}\,,\\
E = \bigcup_{n\ge 1} E_n\,, & E_n = \tau\times\tau(\Hh_n)\,,
\end{eqnarray*}
together with the orientation inherited from $\Hh_n$: so we write \(E_n = E_n^+ \cup E_n^-\) for all $n$, and \(E=E^+\cup E^-\).
Given $h\in\hat\Hh$ we denote by $\Hh_n(h),E_n(h),$ and $E(h)$ the corresponding sets of edges of type $h$.

\begin{lemma}
\label{lem-Gtau}
The approximation graph $G_\tau=(V,E)$ is connected if and only if for all $\varepsilon,\varepsilon'\in\Ee$ with $s(\varepsilon)=s(\varepsilon')$ there is a path in $\hat\Hh$ linking $\varepsilon$ to $\varepsilon'$.
Its set of vertices $V$ is dense in $\Pi_\infty$.
\end{lemma}
\begin{proof}
Let $x,y\in V$, $x\neq y$, and $n$ the largest integer so that $x_i=y_i, i=1, \cdots n-1$, and $x_n\neq y_n$ (so $s(x_n)=s(y_n)$).
Any path in $G_\tau$ linking $x$ to $y$ must contain a subpath linking $x_n$ to $y_n$.
Hence $G_\tau$ is connected iff the above given condition on $\hat\Hh$ is satisfied.

The density of $V$ is clear since any base clopen set for the topology of $\Pi_\infty$: $[\gamma]=\{x\in \Pi_\infty\, :\, x_i=\eta_i, \, i\le n\}$, $\gamma\in\Pi_n$, $n\in\NM$,  contains a point of $V$, namely $\tau(\gamma)$.
\end{proof}

Given an edge $e\in E$ we write $e^\op$ for the edge with the opposite orientation: \(s(e^\op)=r(e), r(e^\op)=s(e)\).
Now our earlier construction \cite{KS10} yields a spectral triple from the data of the approximation graph $G_\tau$.
The $C^\ast$-algebra is $C(\Pi_\infty)$, and it is represented on the Hilbert space $\ell^2(E)$ by:
\begin{equation}
\label{eq-rep}
\pi(f)\psi(e) = f(s(e))\psi(e)\,.
\end{equation}
The Dirac operator is given by:
\begin{equation}
\label{eq-Dirac}
 D \phi(e) = \rho^{-n} \phi(e^\op) \,, \qquad e \in E_n\,.
\end{equation}
The orientation yields a decomposition of $\ell^2(E)$ into \(\ell^2(E^+) \oplus \ell^2(E^-)\). 
\begin{theo}
\label{thm-ST}
$(C(\Pi_{\infty}(\Gg)),\ell^2(E),D)$ is an even spectral triple with $\ZM/2$-grading \(\chi\) which flips the orientation.
Its representation is faithful.
If $\hat\HE$ is sufficiently large ({\it i.e.} satisfies the condition in Lemma~\ref{lem-Gtau})  then its spectral distance $d_s$ is compatible with the topology on $\Pi_{\infty}(\Gg)$, and one has:
\begin{equation}
\label{eq-ds}
d_s(x,y)=c_{xy} \rho^{n_{xy}} + \sum_{n>n_{xy}} \bigl( b_n(x)+b_n(y) \bigr)\rho^n\,, \ \text{\rm for } x\neq y\,,
\end{equation}
where $n_{xy}$ is the largest integer such that $x_i=y_i$ for $i<n_{xy}$, and $b_n(z)=1$ if $\hat\tau(z_n) \neq z_{n+1}$ and $b_n(z)=0$ otherwise, for any $z\in\Pi_\infty$.
The coefficient $c_{xy}$ is the number of edges in a shortest path in $\hat\Hh$ linking $x_{n_{xy}}$ to $y_{n_{xy}}$.
If $\hat\Hh$ is maximal:  \(\hat\HE= \left\{ (\varepsilon,\varepsilon') \in \Ee\times\Ee \, : \, \varepsilon\neq \varepsilon', \; s(\varepsilon)=s(\varepsilon') \right\}\), then $c_{xy}=1$ for all $x,y\in\Pi_\infty$.
\end{theo}
\begin{proof}
The first statements are clear.
In particular the commutator $[D,\pi(f)]$ is bounded if $f$ is a locally constant function.
The representation is faithfull by density of $V$ in $\Pi_\infty$.

If $\hat\HE$ satisfies the condition in Lemma~\ref{lem-Gtau}, then the graph $G(\tau)$ is connected.
It is also a metric graph, with lengths given by $\rho^n$ for all edges $e\in E_n$.
By a previous result (Lemma 2.5 in \cite{KS10}) $d_s$ is an extension to $\Pi_\infty$ of this graph metric, and as $\sum \rho^n < + \infty$, it is continuous and given by equation~\eqref{eq-ds} (by straightforward generalizations of Lemma 4.1 and Corollary 4.2 in \cite{KS10}).
\end{proof}

\subsection{Zeta function}\label{ssec-zeta}
We determine the zeta-function for the triple associated with the above data.
\begin{theo}
\label{thm-zeta}
Suppose that the graph matrix is diagonalizable with eigenvalues $\lambda_j$, $j=1, \ldots p$. 
The zeta-function $\zeta$ extends to a meromorphic function on $\CM$ which is invariant under the translation \(z \mapsto z+\frac{2\pi \imath}{\log \rho}\).
It is given by
\[
 \zeta(z) = \sum_{j=1}^p \frac{C_{\hat\HE}^j }{1-\lambda_j\rho^z} + h(z)
\]
where $h$ is an entire function, and $C_{\hat\HE}^j$ is given in equation~\eqref{eq-CH}.
In particular $\zeta$ has only simple poles which are located at \(\{\frac{\log\lambda_j+2\pi \imath k}{-\log \rho}: k\in \ZM, j=1,\ldots p\}\) with residui given by
\begin{equation}
\label{eq-s0res}
\res(\zeta,\frac{\log\lambda_j+ 2\pi \imath k}{-\log \rho})=\frac{ C_{\hat\HE}^j  \lambda_j}{-\log\rho} \,.
\end{equation}
In particular, the metric dimension is equal to $s_0 = \frac{\log\pf}{-\log \rho}$.
\end{theo}
\begin{proof}
Clearly
\[
\zeta(z) = \sum_{n\ge 1} \rho^{nz} \#E_n
\]
The cardinality of $E_n$ can be computed by summing, over all vertices $v$ at level $n-1$ and all edges $h\in \hat\HE$ with $s^2(h)=v$, the number of paths from level $0$ down to $v$ at level $n-1$:
\[
 \# E_n = \sum_{v\in \Vv} \sum_{h\in \hat\HE} A^{n-1}_{v s^2(h)} \,.
\]
Now since $A$ is diagonalizable, the polynomials $P_{k}^{ij}$ in equation~\eqref{eq-An} are all constant and can be expressed in terms of its (right and left) eigenvectors $R^{k,l},L^{k,l}, 1\le l\le M_k$.
(These vectors were so normalized that the $R^{k,l}$ are the columns the matrix of change of basis which diagonalizes $A$, and the vectors $L^{k,l}$ are the rows of its inverse.)  
So from equation~\eqref{eq-An} we get for all $n$,
\(
\# E_n = \sum_{k = 1}^p C^k_{\hat\Hh} \; \lambda_k^{n}\,,
\)
and hence
\[
 \zeta(z) = \sum_{k = 1}^p C^k_{\hat\Hh} \sum_{n\ge 1} \lambda_{k}^n \rho^{nz}
 = \sum_{k = 1}^p C^k_{\hat\Hh} \frac{\lambda_k\rho^z}{1-\lambda_k\rho^{z}}\,.
\]
Hence $\zeta$ has a simple pole at values $z$ for which $\rho^{z}\lambda_k = 1$, $k=1,\ldots p$.
The calculation of the residui is direct.
\end{proof}
The periodicity of the zeta function with purely imaginary period whose length is only determined by the factor $\rho$ is a feature which distinguishes our self-similar spectral triples from the known triples for manifolds. Note also that $\zeta$ may have a (simple) pole at $0$, namely if $1$ is an eigenvalue of the graph matrix $A.$ 

\begin{rem}
\label{rem-zeta}
In the general case, when $A$ is not diagonalizable, it is no longer true that the zeta-function has only simple poles.
Here the polynomials $P_{k}^{ij}(n)$ in equation~\eqref{eq-An} are non constant (of degree $m_k-1>0$ for $k=2,\ldots p$), and give power terms in the sum for $\zeta(z)$ written in the proof of Theorem~\ref{thm-zeta} (one gets sums of the form \(\sum_{n\ge 1} n^a\lambda_k^n\rho^{nz}\) for integers $a\le m_k-1$).
In this case $\zeta(z)$ has poles of order $m_j$ at \(z=\frac{\log\lambda_j+ 2\pi \imath k}{-\log \rho}\).
\end{rem}

\subsection{Heat kernel}

We derive here the asymptotic behaviour of the trace of the heat kernel
$\TR(e^{-t D^2})$ around $0$.
We give an explicit formula when the graph matrix $A$ is diagonalizable, and explain briefly afterwards (see Remark~\ref{rem-heatkernel}) how the formula has to be corrected when this is not the case.

For $r>0$, $\Re(\alpha)>0$ and $s\in\R$ we define
\begin{equation}
\label{eq-fnper1}
\tilde\f(r,\alpha,s) = \sum_{k=-\infty}^\infty \Gamma(\alpha+\frac{2\pi \imath k}{r})\, e^{2\pi \imath k s} \,.
\end{equation}
In particular,
\begin{equation}
\label{eq-fnper2}
\f_{r,a}(\sigma)  := \frac1r \tilde\f(r,\frac{a}{r},\frac{\sigma}{r})
\end{equation}
is a periodic function of period $r$, with average over a period \(\bar \f_{r,a}=\frac{1}{r} \Gamma(\frac{a}{r})\). 
\begin{theo}
\label{thm-heatkernel}
Consider the above spectral triple for a Bratteli diagram with graph matrix $A$ and parameter $\rho\in(0,1)$.
We assume that $A$ has no eigenvalue of modulus one.
Write its (generalized) eigenvalues $\lambda_j$, for $j=1, \ldots p$.
Let $m_j$ be the size of the largest Jordan block of $A$ to eigenvalue $\lambda_j$, and $P_j$ the polymomial of degree $m_j$ as in equation~\eqref{eq-An}.
Then the trace of the heat kernel has the following expansion as $t\to 0^+$:
\begin{equation}
\label{eq-traceheat}
\TR(e^{-t D^2}) = \sum_{j : |\lambda_j| > 1} 
 P_j \Bigl( \frac{1}{-\log \rho}\frac{d}{d s_j} \Bigr) \,
\f_{-2\log \rho, -s_j \log\rho}({-\log t}) \; t^{- \frac{s_j}{2}}
\  + h(t) \,
\end{equation}
where $s_j = \frac{\log\lambda_j}{-\log \rho}$, 
and $h$ is a smooth function around $0$. 
The leading term of the expansion comes from the Perron-Frobenius eigenvalue, and one has the equivalent
\begin{equation}
\label{eq-heatequiv}
\TR(e^{-t D^2}) \ \asym C^1_{\hat\Hh} \  \f_{-2\log \rho,\log\pf}(-\log t) \ t^{- \frac{s_0}{2}} \,.
\end{equation}
where $s_0= \frac{\log \pf}{-  \log \rho}$ is the spectral dimension as given in Theorem~\ref{thm-zeta}, and $C^1_{\hat\Hh}$ is given in equation~\eqref{eq-CH}.
\end{theo}
\begin{proof}
From equation~\eqref{eq-An} and the proof of Theorem~\ref{thm-zeta} we have,
\(
\# E_n = \sum_{j = 1}^p P_j(n) \lambda_j^n\,,
\)
where $P_j$ is a polynomial of degre $m_j$, for all $n$ greater than or equal to $\max\{m_j : 0\le j \le p\}$ 
(and if $n<\max\{m_j : 0\le j \le p\}$, $P_j$ has to be replaced by a polynomial of degree less than or equal to $m_j$, depending on $n$).
Setting $v = \rho^{-2}$ 
the trace of the heat kernel reads
\begin{equation}
\label{eq-heat-1}
\sum_n \# E_n e^{-t \rho^{-2n}}
= \sum_{j=1}^p \sum_{n=0}^{+\infty}P_j(n) \lambda_j^n \exp(-tv^n)  + g_1(t) \\
\end{equation}
where $g_1(t)$ is a smooth function around zero (the finite sum over $n<\max\{m_k : 0\le k \le p\}$ of terms correcting the formula for $\# E_n$).

First consider an eigenvalue of modulus less than one: $\lambda_j$ with $|\lambda_j|<1$.
We have
\[
 \bigl| P_j(n) \lambda_j^n e^{-tv^n} \bigr| \le \bigl| P_j(n) \lambda_j^{n/2} \bigr|  |\lambda_j|^{n/2} |e^{-tv^n}| \bigr|\le c_j |\lambda_j|^{n/2}
\] 
where $c_j$ is a constant.
So the corresponding series in equation~\eqref{eq-heat-1} is absolutely summable, and therefore eigenvalues of modulus less than 1 do not contribute to the singular behaviour of the trace.

Hence the trace of equation~\eqref{eq-heat-1} can be rewritten as 
\begin{eqnarray*}
\sum_n \# E_n e^{-t \rho^{-2n}}
& = &\sum_{j: |\lambda_j| > 1} \sum_{n=0}^{+\infty}P_j(n) v^{n\frac{s_j}{2}} \exp(-tv^n) + g_2(t) \\
& = & \sum_{j: |\lambda_j| > 1} P_j(\frac1{-\log \rho}\frac{d}{d s_j})
  \sum_{n=0}^{+\infty}v^{n \frac{s_j}{2}} \exp(-tv^n) + g_2(t),
\end{eqnarray*}
where $g_2$ is a smooth function around $0$.

Consider now an eigenvalue of modulus greater than one: $\lambda_j$ with $|\lambda_j|>1$. 
We thus have $\Re(s_j)>0$, and we can write
\[
\sum_n  v^{n \frac{s_j}{2}} e^{-t \rho^{-2n}} =
\sum_{n=-\infty}^{+\infty} v^{n\frac{s_j}{2}} \exp(-tv^n) -
 \sum_{n=1}^{+\infty} v^{-n \frac{s_j}{2}}\exp(-tv^{-n}).
\]
The term \(\sum_{n=1}^{+\infty} v^{-n \frac{s_j}{2}}\exp(-tv^{-n})\) being bounded at $0$ we only need to concentrate on the first sum.
Clearly \(t^{\frac{s_j}{2}}\sum_{n=-\infty}^{+\infty} v^{n\frac{s_j}{2}} \exp(-tv^n)=f(\frac{\log t}{\log v})\) where
\[
f(s) = \sum_{n=-\infty}^\infty v^{(n+s)\frac{s_j}{2}} \exp(-v^{n+s})\,.
\]
By standard arguments this series is uniformly convergent and defines a smooth $1$-periodic function $f$.
It follows that the singular behaviour of $\sum_{n=0}^{+\infty} v^{n\frac{s_j}{2}} \exp(-tv^n)$ as $t\to 0^+$ is given by $f(\frac{\log t}{\log v}) t^{-\frac{s_j}{2}}$.
So we get that the trace reads:
\begin{equation}
\label{eq-heat3}
\sum_n \# E_n e^{-t \rho^{-2n}}
 = \sum_{j: |\lambda_j| > 1} P_j(\frac1{-\log \rho}\frac{d}{d s_j}) f(\frac{\log t}{\log v}) t^{-\frac{s_j}{2}}  + h(t),
\end{equation}
where $h$ is a smooth function around $0$.
And we are left with identifying the function $f$.
Its Fourier coefficients are given by 
\[ 
\hat f_k = \int_{-\infty}^\infty e^{\log(v) \alpha_j x} \exp(-v^x) e^{2\pi \imath k x}dx  
= \frac{\Gamma(\alpha_j+\frac{2\pi \imath k}{\log v})}{\log v},
\]
so we see from equation~\eqref{eq-fnper1} and~\eqref{eq-fnper2} that \(f(s) = \frac{1}{-2\log\rho}\tilde\f(-2\log\rho,\alpha_j,-s)\).
Hence the singular term associated with $\lambda_j$ reads \( f(\frac{\log t}{\log v}) t^{-\frac{s_j}{2}} = \f_{-2\log \rho, -s_j \log \rho}(-\log t)\).
We substitute this back into equation~\eqref{eq-heat3} to complete the proof of equation~\eqref{eq-traceheat}.

Clearly $s_0= \frac{\log \pf}{-\log \rho}$ has the greatest modulus among all the other $s_j$.
Hence the leading term in the expansion comes from the Perron-Frobenius eigenvalue.
Since $\lambda_1=\pf$ is an eigenvalue, $P_1=C^1_{\hat \Hh}$ is the constant polynomial given in equation~\eqref{eq-CH}, which proves equation~\eqref{eq-heatequiv}.
\end{proof} 

We could have determined 
asymptotic expansion of the trace of the heat kernel using the inverse Mellin transform of the function $\zeta(2s)\Gamma( s)$. This is, of course, a lot more complicated than the direct computations. When using the inverse Mellin transform of $\zeta(2s)\Gamma(s)$ 
one can see that the origin of the periodic function $\f$ is directly related to the periodicity of the zeta function and that the appearence of the term $\log t$ in the trace of the heat kernel expansion arises from a simple pole of  $\zeta(s)$ at $s=0$ which amounts to a {\em double} pole of  $\zeta(2s)\Gamma(s)$ at $s=0$.

\begin{rem}
\label{rem-heatkernel}
If $A$ is not diagonalizable, we don't know how to compute contributions of eigenvalues of modulus one, so we assumed $A$ didn't have any in Theorem~\ref{thm-heatkernel}.
But if $A$ is diagonalizable, contributions of such eigenvalues are easily computed. 
Eigenvalues of modulus one, not equal to one, do not contribute to the singular behaviour of the trace.
Only the eigenvalue $\lambda_{j_0}=1$, if present in the spectrum of $A$, gives an extra term. 
And eigenvalues of modulus greater than one contribute as in equation~\eqref{eq-traceheat}, but with $P_j=1$ since $A$ is diagonalizable.
The trace of the heat kernel in this case has the following expansion as $t\to 0^+$:
\begin{equation}
\label{eq-traceheatdiag}
\TR(e^{-t D^2}) = \sum_{j : |\lambda_j| > 1} C^{j}_{\hat\Hh}  \, 
\f_{-2\log \rho, -s_j \log \rho}({-\log t}) \; t^{-\frac{s_j}{2}}
\ +C^{j_0}_{\hat\Hh}\frac{-\log t}{-2\log \rho}  + h(t) \,
\end{equation}
where $s_j=\frac{\lambda_j}{-\log \rho}$, $C^j_{\hat\Hh}$ is given in equation~\eqref{eq-CH}, $j_0$ is the index for the eigenvalue $\lambda_{j_0}=1$ (setting $C^{j_0}_{\hat\Hh}=0$ if $A$ has no eigenvalue equal to $1$) and $h$ is a smooth function around $0$. 
The trace of the heat kernel has therefore the same equivalent as in equation~\eqref{eq-heatequiv}.
\end{rem}

\subsection{Spectral state}\label{ssec-state}
There is a natural Borel probability measure on $\Pi_\infty$. Indeed, due to the primitivity of the graph matrix there is a unique Borel probability measure which is invariant under the action of the groupoid given by tail equivalence. We explain that.

If we denote by $[\gamma]$ the cylinder set of infinite paths beginning with $\gamma$, then invariance under the above mentionned groupoid means that $\mu([\gamma])$ depends only on the length $|\gamma|$ of $\gamma$ and its range,
i.e.\, $\mu([\gamma]) = \mu(|\gamma|,r(\gamma))$.
By additivity we have
\[
\mu([\gamma]) = \sum_{\varepsilon:s(\varepsilon)=r(\gamma)} \mu([\gamma \varepsilon])
\]
which translates into
\[
\mu(n,v) = \sum_{w} A_{vw}\mu(n+1,w)\,.
\]
The unique solution to that equation is
\[
\mu(n,v) = \pf^{-n}R_v
\]
where $R$ is the {\em right} Perron-Frobenius eigenvector of the adjacency matrix $A$,
normalized as in equation~\eqref{eq-normalization}.
So if $\gamma \in \Pi_n$ is a path of length $n$, then \(\mu([\gamma]) = \pf^{-n} R_{r(\gamma)}\).

\begin{theo}
\label{thm-specmeas} All operators of the form $\pi(f)$, $f\in C(\Pi_\infty)$, are strongly regular.
Moreover,  the measure $\mu$ defined on $f\in C(\Pi_\infty)$ by
\[
 \mu(f) : =  \Tt(\pi(f)) 
\]
is the unique measure which is invariant under the groupoid of tail equivalence.
\end{theo}
\begin{proof}
Let $f$ be a measurable function on $\Pi_\infty$, and set 
\[
\mu_n(f)= \frac{\TR_{E_n}\bigl(\pi(f)\vert_{E_n}\bigr)}{\# E_n}=\frac{\sum_{e\in E_n} f(s(e))}{\# E_n}.
\]
To check that the sequence $(\mu_n(f))_{\NM}$ converges it suffices to consider $f$ to be a characteristic function of a base clopen set for the (Cantor) topology of $\Pi_\infty$.
Let $\gamma$ be a finite path of length $|\gamma|<n$ and denote by $\chi_\gamma$ the characteristic function on $[\gamma]$. 
Then $\chi_\gamma (s(e))$ is non-zero if the path $s(e)$ starts with $\gamma$. Given that the tail of the path $s(e)$ is determined by the choice function $\tau$, the number of $e\in E_n$ for which $\chi_\gamma (s(e))$ is non-zero coincides with the number of paths of length $n-|\gamma|-1$ which start at $r(\gamma)$ and end at $s(s(e)_{n-1}) = s^2(h)$, for some $h\in \hat\Hh$.
Hence
\[
\sum_{e\in E_n} \chi_\gamma (s(e)) = \sum_{h\in \hat\HE}
  A^{n-|\gamma|-1}_{r(\gamma)\,s^2(h)} \,.
\]
As noted before in the proof of Theorem~\ref{thm-zeta}, the cardinality of $E_n$ is asymptotically $C^1_{\hat\Hh}\pf^{n}$, so we have
\[
\mu_n (\chi_\gamma)=
\frac{\sum_{e\in E_n} \chi_\gamma(s(e))}{\#E_n} = \pf^{-|\gamma|-1}
\frac{1}{C^1_{\hat\Hh}} \sum_{h\in \hat\HE} R_{r(\gamma)} L_{s^2(h)}
\ (1 + o(1) ) \,.
\]
Set \(U_v = (\pf^{-1}/C^1_{\hat \Hh}) \sum_{h\in \hat\HE} R_{v} L_{s^2(h)})\).
We readily check that $U$ is a (right) eigenvector of $A$ with eigenvalue $\pf$, and since its coordinates add up to $1$, we have $U = R$.
So we get 
\[
\mu_n(\chi_\gamma) = \pf^{-|\gamma|} R_{r(\gamma)} (1+o(1)) \xrightarrow{n\rightarrow +\infty} \pf^{-|\gamma|} R_{r(\gamma)}
= \mu([\gamma])\,.
\]
Now Corollary~\ref{cor-cesar} implies that $\pi(f)$ is strongly regular and $\Tt(\pi(f))=\mu(f)$.
\end{proof}

Below, when we discuss in particular the transversal part of the Dirichlet forms for a substitution tiling, we will encounter operators which are weakly regular, but not strongly regular, and so, a priori, the spectral state is not multiplicative on tensor products of these.
The following lemma will be useful in this case.
\begin{lemma}
\label{lem-stateres}
Consider the above spectral triple for a Bratteli diagram with graph matrix $A$ and parameter $\rho\in(0,1)$.
Let $A\in \Bb(\HS)$ be such that $\bar A_n \stackrel{n\rightarrow \infty}{\sim} e^{\imath n\varphi}$ for some $\varphi\in(0,2\pi)$.
Then
\[
\TR(e^{-t D^2} A) \asym e^{\imath\frac{\phi \log t}{2\log\rho}} {C^1_{\hat\HE}}\,
\f_{-2\log \rho,\log\pf+\imath\varphi}(-\log t) \, t^\frac{\log \pf}{2\log\rho}\,.
\]
In particular, $A$ is weakly regular and $\Tt(A) = 0$.
\end{lemma}
\begin{proof}
As before we set $\sigma = \frac{\log t}{\log v}$, $v = \rho^{-2}$,
$\alpha = \frac{\log\pf}{-2\log \rho}$
so that we have to determine the asymptotic behaviour of
$\sum_{n=1}^\infty e^{\imath n \varphi} v^{\alpha(n+\sigma)} e^{-v^{n+\sigma}}$ when $\sigma\to -\infty$.
Since
$\sum_{n=-\infty}^{-1} e^{\imath n \varphi} v^{\alpha(n+\sigma)} e^{-v^{n+\sigma}}$
is absolutely convergent the sum over $\NM$ has the same asymptotic behaviour than the sum over $\ZM$.
Now
\begin{eqnarray*}
\sum_{n=-\infty}^\infty e^{\imath n \varphi}
v^{\alpha(n+\sigma)} e^{-v^{n+\sigma}} & = &
 e^{-\imath \sigma \varphi} \sum_{n=-\infty}^\infty
v^{(\alpha+\frac{\imath \varphi}{\log v})(n+\sigma)} e^{-v^{n+\sigma}}\\ &=&
e^{-\imath \sigma \varphi}\, \frac1{\log v}\, \tilde \f(\log
v,\frac{\log\pf+\imath \varphi}{\log v},-\sigma) 
\end{eqnarray*}
from which the first statement follows.
$\Tt(A)$ is equal to the mean of the function
$\sigma\mapsto e^{-\imath \sigma \varphi}\,\frac1{\log v}\,\tilde \f(\log v,\frac{\log\pf+\imath \varphi}{\log v},-\sigma)$
which is zero since for all integer $k$ is $\varphi + 2\pi k \neq 0$.
\end{proof}
We now consider the tensor product of two spectral triples associated
with Bratteli diagrams, one with parameter $\rho<1$ and Perron Frobenius eigenvalue $\pf$ the other with
parameter $\rho'<1$ and Perron Frobenius eigenvalue $\pf'$. 
\begin{lemma}
\label{lem-resphi} 
Let $A$ be as in the last lemma, then $\Tt(A\otimes \id) = 0$ if for all
integer $k,k'$ 
\begin{equation}
\varphi + 2\pi k + 2\pi \frac{\log \rho}{\log \rho'}k'\neq 0.
\end{equation}
\end{lemma}  
\begin{proof} By the above general results
$\Tt(A\otimes \id)$ is equal to the mean of the function 
\(s\mapsto e^{\imath \frac{s\varphi}{\log v}}\,\tilde\f(\log v,\frac{\log\pf+\imath\varphi}{\log v},\frac{-s}{\log v})\,\tilde\f(\log v',\frac{\log\pf'}{\log v'},\frac{-s}{\log v'})\) 
devided by the mean of the function
\(s\mapsto \tilde\f(\log v,\frac{\log\pf}{\log v},\frac{-s}{\log v})\,\tilde\f(\log v',\frac{\log\pf'}{\log v'},\frac{-s}{\log v'})\).
The latter is always strictly greater than zero.
By developing the two $\tilde\f$-functions into Fourier series one sees that the former mean is zero
if  $s\mapsto e^{\imath\frac{s\varphi}{\log v}}e^{-\imath\frac{2\pi k s}{\log v}}e^{-\imath \frac{2\pi k' s}{\log v'}}$
is oscillating for all $k,k'$.
\end{proof}
Motivated by this lemma we define:
\begin{defini}
\label{def-resphase} Let $\rho,\rho'\in(0,1)$. We call
$\phi\in (0,2\pi)$ a {\em non resonant phase} (for $(\rho,\rho'$) if  
\begin{equation}
\label{eq-resphase}
\varphi + 2\pi k + 2\pi \frac{\log \rho}{\log \rho'}k'\neq 0 \,, \quad \forall k,k'\in\ZM\,.
\end{equation}
\end{defini}
If $\frac{\log \rho}{\log\rho'}$ is irrational and
for instance $\varphi = 2\pi \frac{\log \rho}{\log \rho'}$, then
$\varphi$ is resonant and, as follows easily from the calculation in
the proof, $\Tt(A\otimes \id)\neq 0$.
In particular, $\Tt(A\otimes \id)\neq \Tt_1(A)\Tt_2(1)$.

\subsection{Quadratic form}\label{ssec-form}
In Section~\ref{ssec-Laplace} we considered a
quadratic form which we specify to the context of spectral triples
defined by Bratteli diagrams. Let $\zeta$ and $s_0$ be the zeta function
with its abscisse of convergence and $\mu$ be the spectral measure on
$\Pi_\infty$ as defined in
Section~\ref{ssec-zeta} and
\ref{ssec-state}. Equations~\eqref{eq-rep} and~\eqref{eq-Dirac} imply, 
for $e\in E_{n}$
\begin{equation}
\label{eq-deltaef0}
[D,\pi(f)] \phi(e) = (\delta_{e}^\alpha f)  \, \phi(e^\op) \,, \quad
\text{\rm with}\quad \delta_e f = \frac{f(r(e)) -f(s(e))}{\rho^n} \,.
\end{equation}
Equation~\eqref{eq-Dirform} therefore becomes
\begin{eqnarray}\nonumber
Q(f,g) &=& \lim_{s\rightarrow s_0^+} \frac{1}{\zeta(s)} \TR \bigl(
|D|^{-s} [D,\pi(f)]^\ast [D,\pi(g)] \bigr) 
\\ \label{eq-Q2} &=& 
 \lim_{s\rightarrow s_0^+} \frac{1}{\zeta(s)} 
\sum_{n\geq 1} \#E_n \rho^{ns} \; q_n(f,g)
\end{eqnarray}
with 
\begin{equation}\label{eq-qn}
q_n(f,g) = \frac{1}{\#E_n}\sum_{e\in E_n} \; \overline{\delta_e f} \;
\delta_e g\,. 
\end{equation}
Note that $ \lim_{s\rightarrow s_0^+} \frac{1}{\zeta(s)} 
\sum_{n\geq 1} \#E_n \rho^{ns} = 1$. We thus have the following simple
result:
\begin{lemma}\label{lem-qn}
$Q(f,g) = \lim_{n\to \infty} q_n(f,g)$ provided the limit exist.
\end{lemma}
The following can be said in general.
\begin{proposi}
\label{prop-form}
The quadratic form $Q$ 
is symmetric, positive definite, and Markovian on 
\(\Dd=\bigl\{f \in L^2_{\RM}(\Pi_\infty,d \mu) \, : \,
Q(f,f) < +\infty \bigr\}\,\).
\end{proposi}
\begin{proof}
The form is clearly symmetric and positive definite. 
Let $\epsilon>0$ and consider a smooth approximation $g_\epsilon$ of
the restriction to $[0,1]$: 
\(g_\epsilon(t)=t\) for \(t\in [0,1]\), \(-\epsilon \le g_\epsilon(t)
\le 1+ \epsilon\) for \(t \in \RM\), and \(0\le  g_\epsilon(t) -
g_\epsilon(t')\le t-t'\) for all \(t\le t'\). 
This last property implies that 
\( |\delta_{e} \ g_\epsilon \circ f | \le |\delta_{e} f | \) so that
 \(Q(g_\epsilon \circ f,g_\epsilon \circ f)\le
 Q(f,f)\). This proves that $Q$ is Markovian. 
\end{proof}
For a precise evaluation of this quadratic form on certain domains 
we have, however, to consider more specific systems.

\subsection{A simple example: The graph with one vertex and two edges}
\label{sec-graph1}
Remember that one difficulty with the Dirichlet form defined by our
spectral triple (equation~\eqref{eq-Dirform} in
Section~\ref{ssec-Laplace}) is that we need to specify a core for the
form. 
We provide here an example where such a core can be suggested with the help of
an additional structure.

We consider the graph $\Gg$ which has one vertex and two
edges, which are
necessarily loops. 
Call one edge $0$ and the other $1$. Let $\lp$ be the edge $0$.
Then it is clear that $\Pi_\infty$ can be identified with  the set of
all $\{0,1\}$-sequences and $\Pi_{\infty \ast}$ with those sequences
which eventually become $0$. 
We consider the spectral triple of Section~\ref{sec-ST-Bratteli}
associated with the graph $\Gg$ and parameter $\rho=\frac12$.
There is not much choice for the horizontal egdes,
$\hat\HE=\{(0,1),(1,0)\}$, nor for $\hat\tau$, $\hat\tau(1) = 0$ and
$\hat\tau(0) = 0$. 
We will look at this system from two different angles justifying two
different cores for the Dirichlet form.

\paragraph{Group structure}
The space of $\{0,1\}$-sequences carries an Abelian group
structure. In fact, if we identify  $\Pi_n$ with $\Z_{2^n}$ using
$\gamma \mapsto \sum_{i=1}^n \gamma_i 2^{i-1}$ then we can write
$\Pi_\infty$ as the inverse limit group 
$$\Pi_\infty = \lim_{\longleftarrow} \left
  (\Z_{2^n}\stackrel{\pi_n^{n+1}}{\longleftarrow}
  \Z_{2^{n+1}}\right)$$ 
where $\pi_n^{n+1}$ maps $m\: mod \: 2^{n+1}$ to  $m\: mod \:2^n$.

It follows that $C(\Pi_\infty)$ is isomorphic to the group
$C^*$-algebra of the Pontrayagin dual  
\[
 \hat\Pi_\infty = \lim_{\longrightarrow} \left
   (\hat\Z_{2^n}\stackrel{\hat\pi_n^{n+1}}{\longrightarrow}
   \hat\Z_{2^{n+1}}\right)\,. 
\]
This suggest that a reasonable choice for the domain of the form would
be the (algebraic) group algebra $\CM \hat\Pi_\infty$ 
whose elements are finite linear
combinations of elements of $\hat\Pi_\infty$. Now $\CM \hat\Pi_\infty$
corresponds to  the subalgebra  $C_{l.c.}(\Pi_\infty)\subset
C(\Pi_\infty)$ of locally constant functions. We have
$\lim_{n\to \infty} q_n(f,g)=0$ if $f$ and $g$ are locally constant
and so by Lemma~\ref{lem-qn} $Q$ exists on that core. However, 
$Q$ is identically $0$ on the domain defined by the core. So the
algebraic choice of the domain which was motivated by the group
structure is not very interesting.

\paragraph{Embedding into $S^1$}
Note that  $\Pi_n$ can be identified with $\hat \Z_{2^n} $ via the map
\(\gamma \mapsto \exp(2\pi\imath \sum_{j=1}^n  \gamma_j
2^{-j+1}(\cdot))\) so that $ \Pi_{\infty\ast}\cong\hat\Pi_\infty =
\{\exp(2\pi\imath p (\cdot)):p\in\Z[\frac12]\cap [0,1)\}$.  
This suggest to view $\Pi_{\infty\ast}$ (via
$\gamma \mapsto
\exp(2\pi\imath \sum_{i\geq 1}  \gamma_i 2^{-i+1})\in S^1 $) as a
dense subset of $S^1$. 
The inclusion $\Pi_{\infty\ast}\hookrightarrow S^1$
extends to a continuous surjection $\eta:\Pi_\infty
\to S^1$ which is almost everywhere one-to-one 
Furthermore the push forward of the measure $\mu$ on $\Pi_\infty$ 
is the Lebesgue measure on $S^1$. 
Hence the pull back $\eta^*$ induces an isometry between $L^2(S^1)$
and $L^2(\Pi_\infty)$. 
It follows that $\eta^*(C(S^1))$ is dense in $L^2(\Pi_\infty)$.
This suggests to take as core for the quadratic form 
the pull backs of trigonometric polynomials
over $S^1$. Now one sees from Equ.~\ref{eq-qn} that $\lim_{n\to
  \infty} q_n(f,g)$ exists, provided $\rho = \frac12$. In fact,
$\lim_{n\to  \infty} q_n(f,g) = 
\langle \frac{\partial f}{\partial x},\frac{\partial g}{\partial
  x}\rangle$
showing that $Q$ has
infinitesimal generator equal to the standard Laplacian on
$S^1$.

\subsection{Telescoping}
There is a standard equivalence relation among Bratteli diagrams which is generated by isomorphisms and so-called telescoping.
Since we are looking at stationary diagrams we consider stationary telescopings only.
Then the following operations generate the equivalence relation we consider:
\begin{enumerate}
\item {\it Telescoping}: 
Given the above data built from a graph $\Gg=(\Vv,\Ee)$, and a positive integer $p$, we consider a new graph $\Gg^p:= (\Vv^p,\Ee^p)$ with the same vertices: $\Vv^p=\Vv$, and the paths of length $p$ as edges: $\Ee^p=\Pi_p(\Gg)$. 
The corresponding parameter is taken to be $\rho_p=\rho^p$.

\item {\it Isomorphism}: Given two graphs as above $\Gg=(\Vv,\Ee)$, $\Gg'=(\Vv',\Ee')$, we say that the corresponding stationary Bratteli diagrams are isomorphic if there are two bijections $\Vv \rightarrow \Vv'$, $\Ee \rightarrow \Ee'$ which intertwine the range and source maps.
We need in this case that the associated parameters be equal, and the sets of horizontal edges isomorphic (through a map which intertwines the range and source maps).
\end{enumerate}
We show now that this equivalence relation leaves the properties of the associated spectral triple invariant: 
\begin{enumerate}[(i)]
\item The zeta functions are equivalent, so have the same spectral dimension: $s_0=\frac{\log\pf}{-\log\rho}$;

\item The spectral measures are both equal to the invariant probability measure $\mu$ on $\Pi_\infty$;

\item Both spectral distances generate the topology of $\Pi_\infty$ (provided $\hat\Hh$ is large enough as in Lemma~\ref{lem-Gtau}), and are  furthermore Lipschitz equivalent.
\end{enumerate}
The invariance under isomorphism is trivial.
We explain briefly how things work under telescoping.
The horizontal edges for $\Gg^p$ are given as for $\Gg$, by the corresponding subset
\[
\hat\HE^p \subseteq \left\{ (\varepsilon,\varepsilon') \in \Ee^p\times\Ee^p \, : \, \varepsilon\neq \varepsilon', \; s(\varepsilon)=s(\varepsilon') \right\} \,,
\]
and so we have the identifications
\[
\Hh_n^p \cong \bigcup_{i=0}^{p-1} \Hh_{np+i}\,,  \qquad E_n^p\cong \bigcup_{i=0}^{p-1} E_{np+i}
\]
which allows us to determine the approximation graph $G_\tau^p=(V^p,E^p)$, and yields a unitary equivalence \(\ell^2(E) \cong \ell^2(E^p)\).
We identify the two Hilbert spaces $\ell^2(E) \cong \ell^2(E^p)$ and the representations $\pi\cong \pi_p$, while the Dirac operators satisfy:
\[
D_p = W^\dagger D W \quad \text{with } \quad
W : \left\{ \begin{array}{ccl}
\ell^2(E) & \rightarrow & \ell^2(E^p) \\
\delta_e & \mapsto & \rho^{-\frac{k}{2}} \delta_e \,, \ \text{\rm for } e\in E_n, \text{\rm with } n= k \mod p\,.
\end{array}
\right.
\]
From the inequalities \( \id \le W \le \rho^{-p} \id\), we deduce that the zeta functions are equivalent, and that both spectral triple have the same spectral dimension $s_0$, and give rise to the same spectral measure $\mu$.
By Theorem~\ref{thm-ST}, both Connes distances generate the topology of $\Pi_\infty=\Pi_\infty^p$, provided $\hat\Hh$ is large enough.
Let us call $d_s^p$ the spectral metric associated with $\Gg^p$, with corresponding coefficients $n^p_{xy},c_{xy}^p, b^p_n$ as in equation~\eqref{eq-ds}.
Writing $n_{xy}=pn_{xy}^p +k_{xy}$, for some $k_{xy}\le p-1$, we have
\[
d_s(x,y) =  c_{xy} \rho^{k_{xy}} (\rho^p)^{n^p_{xy}} + 
\sum_{n> n^p_{xy}} (\rho^p)^n \sum_{k=0}^{p-1} \bigl( b_{np+k} (x) + b_{np+k} (y) \bigr) \rho^k  \,,
\]
Now we see that \(b_{np+k}(z) = 1 \Rightarrow b_n^p(z)=1\), while if \(b_{np+k}(z)=0\) for all $k=0, \cdots p-1$, then $b_n^p(z)=0$ too, so that one has \( b_n^p(z) \le \sum_{k=0}^{p-1} b_{np+k}(z) \le p b_n^p(z)\).
We substitute this back into the previous equation to get the Lipschitz equivalence:
\[
c_p \ d_s^p (x,y) \  \le  d_s(x,y) \le p \rho^p C_p \ d_s^p(x,y) \,,
\]
with $c_p$, $C_p$, the respective min and max of $c_{xy}/c^p_{xy}$ (which only depends on $\Hh$ and $p$).

\section{Substitution tiling spaces}
\label{sec-tilings}
Bratteli diagrams occur naturally in the description of substitution
tilings. The path space of the Bratteli diagram defined by the
substitution graph has been used to describe the 
transversal of such a
tiling \cite{Forrest,Kel95}. As we will first show, an extended version can also be used to
describe a dense set of the continuous hull $\Omega_\Phi$ of the tiling and
therefore we will employ it and the construction of the previous
section to construct a spectral triple for $\Omega_\Phi$. 
We then will have a closer look at the Dirichlet form defined by the
spectral triple. Under the assumption that the dynamical spectrum of
the tiling is purely discrete we can identify a core for the Dirichlet.
We can then also compute explicitely the
associated Laplacian.

\subsection{Preliminaries}
\label{ssec-tilingbasics}
We recall the basic notions of tiling theory, namely
tiles, patches, tilings of the Euclidean space $\RM^d$, and substitutions. 
For a more detailed presentation in particular of substitution tilings
we refer the reader to \cite{Grunbaum}.  
A {\em tile} is a compact subset of $\RM^d$ that is homeomorphic to a
ball. It possibly carries a decoration (for instance its collar).
A {\em tiling} of $\RM^d$ is a countable set of tiles $(t_i)_{i\in I}$
whose union covers $\RM^d$ and with pairwise disjoint interiors.  
Given a tiling $T$, we call a {\em patch} of $T$, any set of tiles in $T$ which covers a bounded and simply connected set. 
A {\em prototile} (resp. {\em protopatch}) is an equivalence class of tiles (resp. patches) modulo translations.
We will only consider tilings with finitely many prototiles and for which there are only finitly many protopatches containing two tiles (such tilings have Finite Local Complexity or FLC).

The tilings we are interested in are constructed from a (finite) prototile set $\Aa$ and a substitution rule on the prototiles. 
A substitution rule is a decomposition rule followed by a rescaling, i.e.\ each prototile is decomposed into smaller tiles, which, when stretched by a common factor $\theta>1$ are congruent to some prototiles. 
We call $\theta$ the  {\em dilation factor} of the substitution.
The decomposition rule can be applied to patches and whole tilings, by simply decomposing each tile, and so can be the substitution rule when the result of the decomposition is stretched 
by a factor of $\theta$. 
We denote the decomposition rule by $\dec$ and the substitution rule by $\Phi$. In particular we have, for a tile $t$, $\dec (t+a) = \dec (t) + a$ and $\Phi (t+a) = \Phi (t) + \theta a$ for all $a \in \RM^d$. 
See Figure~\ref{fig-chair} for an example in $\RM^2$.
\begin{figure}[htp]
\begin{center}
\includegraphics[scale=0.6]{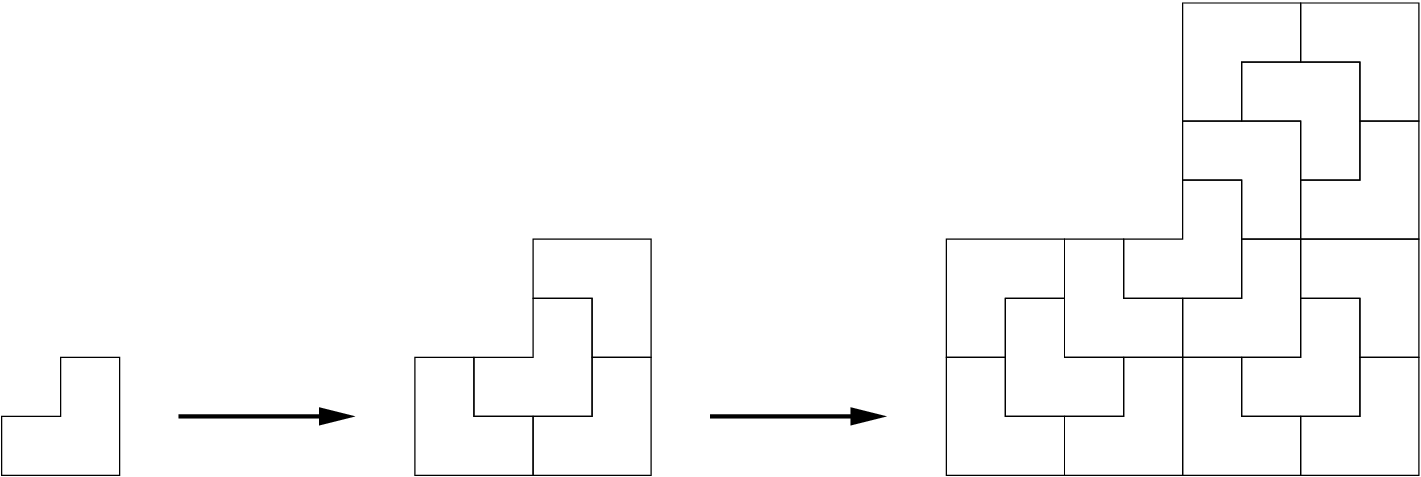}
\caption{{\small A process of inflation and substitution (chair tiling). A whole tiling of~$\RM^2$ can be obtained as a fixed point of this map.}}  
\label{fig-chair}
\end{center}
\end{figure}

A patch of the form $\Phi^n(t)$, for some tile $t$, is called an
{\em $n$-supertile}, or $n$-th order supertile.  
A rescaled tile $\theta^{n} t$ will be called a {\em level $n$ tile} but also, if $n=-m<0$, an {\em $m$-microtile}, or $m$-th order microtile. 

A substitution defines a tiling space $\Omega_\Phi$: the set of all
tilings $T$ with the property that any patch of $T$ occurs in a
supertile of sufficiently high order.  

We will assume that the substitution is {\em primitive} and {\em
  aperiodic}: there exists an integer $n$ such that any $n$-supertile
contains  tiles of each type and all tilings  of $\Omega_\Phi$ are aperiodic. 
This implies that by inspection of a large enough but finite patch
around them the tiles of $\Omega_\Phi$ can be grouped into supertiles
(one says that $\Phi$ is recognizable) so that $\dec$ and $\Phi$
are invertible. 
In particular, $\Phi$ is a homeomorphism of $\Omega_\Phi$ if the
latter is equipped with the standard tiling metric \cite{AP}. 

We may suppose that the substitution {\em forces the border}
\cite{Kel95}.  
The condition says that given any tile $t$, its $n$-th substitute does
not only determine the $n$-supertile $\Phi^n(t)$, but also all tiles
that can be adjacent to it.  
This condition can be realized for instance by considering decorations
of each types of tiles, and replacing $\Aa$ by the larger set of
collared prototiles.  

There is a canonical action of $\RM^d$ on the tiling space
$\Omega_\Phi$, by translation, which makes it a topological
dynamical system.  
Under the above assumptions, the dynamical system $(\Omega_\Phi,
\RM^d)$ is minimal and uniquely ergodic. 
The unique invariant and ergodic probability measure on
$\Omega_\Phi$ will be denoted $\mu$. 

A particularity of tiling dynamical system is that they admit
particular transversals to the $\RM^d$-action.  
To define such a transversal $\Xi$, 
we associate to each prottoile a particular 
point, called its {\em
  puncture}. 
Each level $n$ tile being
similar to a unique proto-tile we may then associate to the level $n$ tile the
puncture which is the image of the puncture of the
proto-tile under the similarity. 
The transversal\footnote{Sometimes $\Xi$ is referred to as the {\em canonical transversal} or the {\em discrete hull}} $\Xi$ is the subset of tilings $T\in\Omega$ which
has the puncture of one of its tiles at the origin of $\RM^d$. 
The measure $\mu$ induces an invariant probability measure on $\Xi$
which gives the frequencies of the tiles
and patches.


\subsection{Substitution graph and the Robinson map}
\label{ssec-robmap}
The {\em substitution matrix} of the substitution $\Phi$ is the
matrix with coefficients $A_{ij}$ equal to the number of tiles of type
$t_i$ in $\Phi(t_j)$.  
The graph $\Gg$ of Section~\ref{sec-ST-Bratteli} underlying our
constructions will be here the {\em substitution graph}: the
graph whose graph matrix is the substitution matrix. 
More precisely, its vertices $v\in\Vv$ are in one-to-one
correspondence with the prototiles and we denote with $t_v$ the
prototile corresponding to $v\in\Vv$, {\it i.e.}\, the prototile set reads \( \Aa = \{ t_v : v \in \Vv\}\).
Between the vertices $u$ and $v$ there are $A_{uv}$ edges
corresponding to the $A_{uv}$ different occurrences of tiles of type
$t_u$ in $\Phi(t_v)$. 
Here we call $u$ (or $t_u$) the source, and $v$ (or $t_v$) the range
of these edges. 
Notice that the Perron-Frobenius eigenvalue of $A$ is the $d$-th power
of the dilation factor $\theta$: \(\pf=\theta^d\). 
The asymptotics of the powers of $A$ are given by
equation~\eqref{eq-An} as before. 
The coordinates of the left and right Perron-Frobenius eigenvectors
$L,R$, 
are now related to
the volumes and the frequencies of the prototiles as follows: for all
$v\in\Vv$ we have
\begin{equation}
\label{eq-LR}
\freq(t_v) = R_v\,,
\qquad 
\vol(t_v) = L_v \,
\end{equation}
where $\freq(t_v)$ is the frequency and $\vol(t_v)$ the volume of
$t_v$, the volume being normalized as in equation~\eqref{eq-normalization} so that the average volume of a
tile is $1$.  

Given a choice of punctures to define the transversal $\Xi$ of
$\Omega_\Phi$ there is a map  
\[
\Rob:\Xi \to \Pi_\infty(\Gg)
\]
onto the set of half-infinite paths in $\Gg$. 
Indeed, given a tiling $T\in\Xi$ (so with a puncture at the origin) and
an integer $n\in \NM$, we define: 
\begin{itemize}
\item $v_n(T)\in \Vv$ to be the vertex corresponding to the prototile
  type of the tile in $\Phi^{-n}(T)$ which contains the origin;  
\item $\varepsilon_n(T)\in \Ee$ to be the edge corresponding to the occurrence
  of $v_{n-1}(T)$ in $\Phi(v_{n}(T))$. 
\end{itemize}
Then $\Rob(T)$ is the sequence $(\varepsilon_n(T))_{n\geq 1}$. 
We call $\Rob$ the Robinson map as it was first defined for the
Penrose tilings by Robinson, see \cite{Grunbaum}.

\begin{theo}[\cite{Kel95}]
\label{thm-homeo}
$\Rob$ is a homeomorphism.
\end{theo}
We extend the above map $\Rob$ to the continuous hull $\Omega_\Phi$. 
The idea is simple: the definition of $\Rob$ makes sense provided the origin lies in a single tile but becomes ambiguous as soon as it lies in the common boundary of several tiles.
We will therefore always assign that boundary to a unique tile in the following way.

We suppose that the boundaries of the tiles are sufficiently regular so that there exist a vector $\vec v\in\RM^d$ such that for all points $x$ of a tile $t$ either 
\(\exists \eta>0, \forall \epsilon \in(0,\eta): x+\epsilon \vec v \in t\) 
or 
\(\exists \eta>0, \forall\epsilon \in (0,\eta): x+\epsilon \vec v \notin t\). 
This is clearly the case for polyhedral tilings.
We fix such a vector $v$.
Given a prototile $t$ (a closed set) we define the half-open prototile $[t)$ as follows:
\[
[t) :=\big\{ x\in t \, : \, \exists \eta>0\:\forall
\epsilon \in[0,\eta): x+\epsilon \vec v \in t \big\}.
\]
It follows that any tiling $T$ gives rise to a partition of $\RM^d$ by
half-open tiles.
We extend the Robinson map to
\[
\Rob: \Omega_\Phi\to\Pi_{-\infty,+\infty}
\]
where $\Pi_{-\infty,+\infty}$ is the space of bi-infinite sequences over $\Gg$ using half-open proto-tiles as follows.
For $n\in \ZM$ we define
\begin{itemize}
\item  $v_n(T)\in \Vv$ to be the vertex corresponding to the prototile type of the half open tile in $\Phi^{-n}(T)$ which contains the origin.
So $v_n(T)$ corresponds to
 \begin{itemize}
 \item the $n$-th order (half-open) supertile in $T$ containing the origin, for $n> 0$,
 \item $n$-th order (half-open) microtile in $\dec^{-n}(T)$ containing the origin, for $n \le 0$;
\end{itemize}
\item $\varepsilon_n(T)\in \Ee$ to be the edge corresponding to the occurrence of $v_{n-1}(T)$ in $\Phi(v_{n}(T))$.
\end{itemize} 
And we set \(\Rob(T)\) to be the bi-infinite sequence $\Rob(T)=\bigl(\varepsilon_n(T)\bigr)_{n\in\ZM}$. 

\begin{rem}
\label{rem-brat2}{\em
As in Remark~\ref{rem-brat1} we can see this construction as a Bratteli diagram, but the diagram is bi-infinite this time.
There is a copy of $\Vv$ at each level $n\in \ZM$ and edges of $\Ee$ between levels $n$ and $n+1$.
Level $0$ corresponds to prototiles, level $1$ to supertiles and level $n>1$ to $n$-th order supertiles, while level $-1$ corresponds to microtiles and level $n<-1$ to $n$-th order microtiles.
For the ``negative'' part of the diagram, we can alternatively consider the reversed substitution graph $\tilde\Gg=(\Vv,\tilde\Ee)$ which is $\Gg$ with all orientations of the edges reversed. 
The graph matrix of $\tilde{\Gg}$, is then the transpose of the substitution matrix: $\tilde{A}=A^T$.
So for $n\le0$, there are edges of $\tilde \Ee$ between levels $n$ and $n-1$: there are $\tilde{A}_{uv}=A_{vu}$ such edges linking $u$ to $v$.
}
\end{rem}

As for Theorem~\ref{thm-homeo} one proves, using the border forcing condition, that $\Rob$ is injective. 

Given a path $\xi \in \Pi_{-\infty,+\infty}$, and $m < n \in \ZM\cup\{\pm \infty\}$, we denote by
\(\xi_{[m,n]}\),\(\xi_{(m,n]}\),\(\xi_{[m,n)}\) and \(\xi_{(m,n)}\), its restrictions from level $m$ to $n$ (with end points included or not).
Also $\xi_n$ will denote its $n$-th edge, from level $n$ to level $n+1$, $n\in\ZM$.
We similarly define $\Pi_{m,n}$ (with end points included).
For instance $\Pi_{0,+\infty}$ is simply $\Pi_\infty=\Pi_\infty(\Gg)$.

We say that an edge $e\in\Ee$ is {\em inner} if it encodes the position of a tile $t$ in the supertile $p$ such that
\(\exists\eta>0, \forall\epsilon\in[0,\eta):t+\epsilon \vec v\in p\).
This says that the occurence of $t$ in $p$ does not intersect the open part of the border of $p$.

It is not true that $\Rob$ is bijective but we have the following.
\begin{lemma}
\label{inedges}
$X:=\mbox{\rm im} \Rob$ contains the set of paths $\xi\in\Pi_{-\infty,+\infty}$ such that $\xi_{(-\infty,0]}$ has 
infinitely many inner edges.
\end{lemma}
\begin{proof}
Recall the following: If $(A_n)_n$ is a sequence of subsets of $\RM^d$
such that $A_{n+1}\subset A_{n}$ and $\mbox{diam}(A_n)\to 0$ then there
exists a unique point $x\in\RM^d$ such that $x\in \bigcap_n \overline{A_n}$. 

By construction, for any tiling $T$, one has $0\in \bigcap_{n\leq 0}[v_n(T))$.
Hence $\xi=\Rob(T)$ whenever $\bigcap_{n\leq 0}[s(\xi_n))\neq\emptyset$,
where $[s(\xi_0))$ is the standard representative for the half-open prototile of type $s(\xi_0)$ and $[s(\xi_n))$  the half-open $n$-th order microtile of type $s(\xi_n)$ in $[s(\xi_0))$ which is encoded by the path $\xi_{[n,0]}$. 

Suppose that $\xi_{n}$ is inner, then $[s(\xi_{n}))$ does not lie at the open border of $[r(\xi_n)=s(\xi_{n+1}))$.
Hence 
\([s(\xi_{n}))\cap [s(\xi_{n+1}))= [s(\xi_{n})]\cap [s(\xi_{n+1}))\)
where $[s(\xi_{n})]$ is the closure of $[s(\xi_n))$.
Suppose that infinitely many edges of $\xi_{(-\infty,0]}$ are inner, then 
\[
\bigcap_{n<0:   \xi_{n}\mbox{ \small inner}} [s(\xi_{n})]
\subset 
\bigcap_{n < 0}[s(\xi_{n+1}))
\]
showing that the r.h.s.\ contains an element, and hence $\xi\in\mbox{\rm im}\,\Rob$.
\end{proof}

\begin{coro}
The set $X$ is a dense and shift invariant subset of $\Pi_{-\infty,+\infty}$.
\end{coro}
\begin{proof}
Shift invariance is clear.
Denseness follows immediately from Lemma~\ref{inedges}.
\end{proof}

In particular, for $n\in\NM$, each element of $\Pi_{-n,n}$ can be the middle part of a sequence in
$\Rob(\Omega_\Phi)$, that is, for all $\gamma\in \Pi_{-n,n}$ there exists $T\in\Omega_\Phi$ such that $\Rob(T)_{[-n,n]} = \gamma$.
\begin{rem}
\label{rem-Robinson}{\em
For $v\in\Vv$, let $\Pi_{-\infty,\infty}^v$ be the set of bi-infinite paths which pass through $v$ at level $0$, and set $X^v = X\cap \Pi_{-\infty,\infty}^v$. 
Then $\Rr$ yields a bijection between $\Xi_{t_v}\times [t_v)$ and $X^v$, where $t_v$ is the prototile corresponding to $v$ and $\Xi_{t_v}$ its acceptance domain (the set of all tilings in $\Xi$ which have $t_v$ at the origin).

Notice that $\Pi_{-\infty,0}$ can be identified with $\Pi_\infty(\tilde\Gg)$, where $\tilde\Gg$ is the graph obtained from $\Gg$ by reversing the orientation of its edges: one simply reads paths backwards, so follows the edges along their opposite orientations.
One sees then, that the Robinson map yields a homeomorphism between
$\Xi_{t_v}$ and $\Pi_{0,+\infty}^v=\Pi_\infty^v$, and a map with dense
image from $[t_v)$ into $\Pi_{-\infty,0}^v=\Pi_\infty^v(\tilde \Gg)$.
}
\end{rem}

\subsection{The transversal triple for a substitution tiling}
\label{ssec-trST}
Our aim here is to construct a spectral triple for the transversal $\Xi$.
We apply the general construction of Section~\ref{sec-ST-Bratteli} to the substitution graph $\Gg=(\Vv,\Ee)$. We may suppose\footnote{This can always be achieved by going over to a power of the substitution.} that the substitution has a fixed point $T^*$ such that $\RM^d$ is covered by the union over $n$ of the $n$-th order supertiles of $T^*$ containing the origin. 
It follows that $\Rob(T^*)$ is a constant path in $\Pi_{-\infty,+\infty}(\Gg)$, that is, the infinite repetition of a loop edge of $\Gg$ which we choose to be $\varepsilon^*$.
We then fix $\tau$, take $\rho=\rtr$ as a parameter, and choose a subset
\[
\hat\Hh_{tr}\subset \Hh(\Gg)= \left\{ (\varepsilon,\varepsilon') \in \Ee\times\Ee \, : \, \varepsilon\neq \varepsilon', \; s(\varepsilon)=s(\varepsilon') \right\}\,
\]
which we suppose to satisfy the conditions of  Lemma~\ref{lem-Gtau}: if $s(\varepsilon) = s(\varepsilon')$ there is a path of edges in $\hat\Hh_{tr}$ linking $\varepsilon'$ with $\varepsilon'$.
The horizontal edges of level $n\in\NM$ are then given by
\[
\Hh_{{tr},n}= \Bigr\{ (\eta \varepsilon,\eta \varepsilon') \, : \, \eta \in \Pi_{n-1}(\Gg), (\varepsilon,\varepsilon')\in \hat\Hh_{tr}
\Bigr\} \subset \Pi_n(\Gg) \times \Pi_n(\Gg) \,.
\]
They define the transverse approximation graph \(G_{tr,\tau}=(V_{tr},E_{tr})\) as in Section~\ref{sec-ST-Bratteli}
\begin{eqnarray*}
V_{tr} = \bigcup_n V_{tr,n}\,, & V_{tr,n} = \tau (\Pi_n(\Gg)) \subset \Pi_\infty^\ast(\Gg)\,,\\
E_{tr} = \bigcup_n E_{tr,n}\,, & E_{tr,n} = \tau\times\tau(\Hh_{tr,n})\,,
\end{eqnarray*}
together with the orientation inherited from $\hat\Hh_{tr}$: so \(E_{tr,n} =
E_{tr,n}^+ \cup E_{tr,n}^-\) for all $n\in\NM$, and
\(E_{tr}=E_{tr}^+\cup E_{tr}^-\). We also write $E_n(h) =  \tau\times\tau(\Hh_{{tr},n}(h))$ where, if $h=(\varepsilon,\varepsilon')$, then $\Hh_{{tr},n}(h) = \{(\eta\varepsilon,\eta\varepsilon'):\eta\in\Pi_{n-1}(\Gg)\}$.
By our assumption on $\hat\Hh$ the
approximation graph $G_{tr,\tau}$ is connected, and its vertices are
dense in $\Pi_\infty(\Gg)$.

An edge $h\in \hat\Hh_{tr}$ has the following interpretation:
The two paths $\tau(s(h))$ and $\tau(r(h))$ have the same source vertex, say $v_0$, they differ on their first edge and then, at some minimal $n_h>0$, they come back together coinciding for all further edges. This is a consequence of the property of $\tau$. Let us denote the vertex at which the two edges come back together with $v_h$. Neglecting the part after that vertex we obtain a pair $(\gamma,\gamma')$ of paths of length $n_h$ which both start at $v_0$ and end at $v_h$.
Reading the definition of the Robinson map $\Rob$ backwards we see that the pair $(\gamma,\gamma')$
describes a pair of tiles $(t,t')$ of type $v_0$ in an $n_h$-th order supertile of type $v_h$.
Of importance below will be the vector  $r_h\in \R^d$ of translation from $t$ to $t'$.
  
The interpretation of an edge $e\in E_{tr,n}(h)$ (so an edge of type $h$) is similar, except that the paths $\tau(s(e))$ and $\tau(r(e))$  coincide up to level $n$ and meet again at level $n+n_h$. In particular $e$ describes a pair of $n$-th order supertiles $(t,t')$ of type $v_0$ in an $(n+n_h)$-th order supertile of type $v_h$. 
If one denotes by $r_e\in \RM^d$ the translation vector between $t$ and $t'$ then, due to the selfsimilarity, one has:
\begin{equation}
\label{eq-transtr}
r_e= \theta^{n} r_h.
\end{equation}
See Figure~\ref{fig-chair3} for an illustration.
\begin{figure}[htp]
\begin{center}
\includegraphics[scale=0.35]{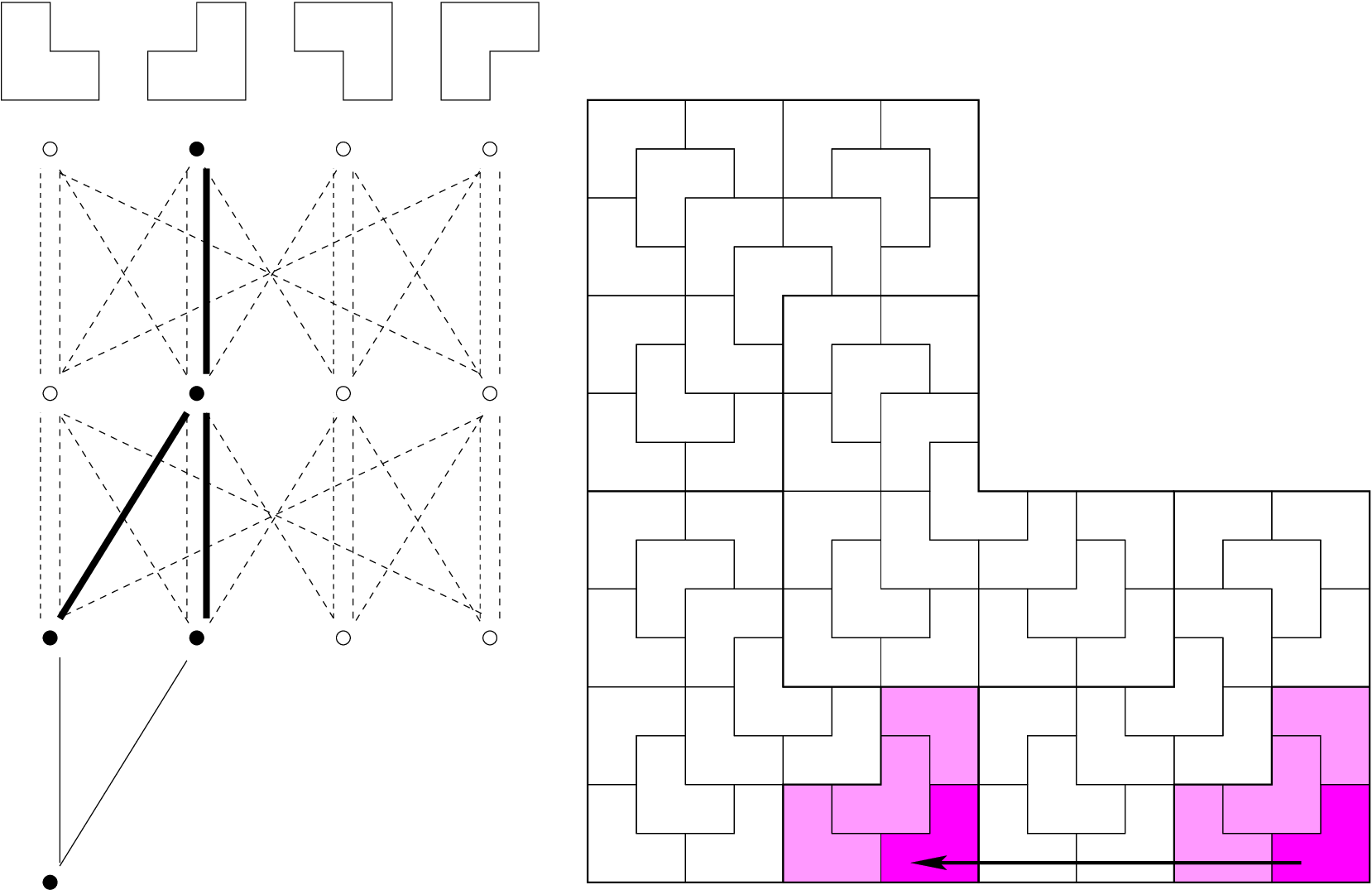}
\caption{{\small A doubly pointed pattern associated with a horizontal arrow $e\in E_{tr,3}(h)$.
The arrow represents the vector $r_e$.
Here $n=2$ (the paths have lengths $2$), and $n_h=1$ (the paths join further down at level $n+n_h=3$).}} 
\label{fig-chair3}
\end{center}
\end{figure}

Theorem~\ref{thm-ST} provides us with a spectral triple for the algebra $C(\Pi_\infty(\Gg))$.
We adapt this slightly to get a spectral triple for $C(\Xi)$.
Since the $n$-th order supertiles of $T^*$ on $0$ eventually cover $\RM^d$,  
$\Rob$ identifies $\Pi_\infty^*(\Gg)$ with the translates of $T^*$
which belong to $\Xi$. We may thus consider the spectral triple
$(C(\Xi),\HS_{tr},D_{tr})$ (which depends on $\rtr$ and the choices
for $\tau$ and $\HE$)) with representation and Dirac operator defined
as in equations~\eqref{eq-rep} and~\eqref{eq-Dirac} by: 
\[
\HS_{tr}=\ell^2(E_{tr}) \,,\quad 
\pi_{tr}(f)\phi(e) = f\bigl( \Rob^{-1}(s(e)) \bigr)\phi(e)\,, \quad 
D_{tr} \,\phi(e) = \rtr^{-n} \phi(e^\op)\,, \  e\in E_{tr,n}\,.
\]
We call it the {\em transverse spectral triple} of the substitution tiling.
By Theorem~\ref{thm-ST} it is an even spectral triple with grading $\chi$ (which flips the orientation).
Also, since $\Hh_{tr}$ satisfies the hypothesis of Lemma~\ref{lem-Gtau} as noted above, the Connes distance induces the topology of $\Xi$.
By Theorems~\ref{thm-zeta} and~\ref{thm-specmeas} the transversal spectral triple has metric dimension $s_{tr}=\frac{d\log(\theta)}{-\log(\rtr)}$, and its spectral measure is the unique ergodic measure on $\Xi$ which is invariant under the tiling groupoid action. 

For $v\in\Vv$, we will also consider the spectral triple $(C(\Xi_{t_v}), \HS_{tr}^v, D_{tr})$ for \(\Xi_{t_v}=\Rob^{-1} \bigl(\Pi_\infty^v(\Gg) \bigr)\): the acceptance domain of $t_v$ (see Remark~\ref{rem-Robinson}).
We call it the {\em transverse spectral triple for the prototile $t_v$}.
It is obtained from the transverse spectral triple by restriction to the Hilbert space $\HS_{tr}^v =\ell^2(E_{tr}^v)$ where $ E_{tr}^v$ are the horizontal edges between paths which start on $v$.
This restriction has the effect that
\[ 
\zeta_v = R_v \zeta + reg.\,,
\]
that is, up to a perturbation which is
regular at $s_{tr}$, the new zeta function is
$R_v=\freq(t_v)$ times the old one.
It hence has the same abscissa of convergene, $s_{tr}^v=s_{tr}$, 
but its residue at $s_{tr}$ is $\freq(t_v)$ times the old one.   

Like for $\Xi$ the Connes distance induces the topology of
$\Xi_{t_v}$. Finally the spectral measure $\mu_{tr}^v$ is the
restriction to $\Xi_{t_v}$ of the invariant measure on $\Xi$, normalized so that the total measure of $\Xi_{t_v}$ is $1$.
The spectral triple for $\Xi$ is actually the direct sum over
$v\in\Vv$ of the spectral triples for $\Xi_{t_v}$. 

We are particularily interested in the quadratic form defined formally by 
\begin{multline}
Q_{tr}(f,g) = \Tt_{tr}([D_{tr},\pi_{tr}(f)]^*[D_{tr},\pi_{tr}(g)]) \\= 
\lim_{s\downarrow s_{tr}} 
\frac{\TR_{\HS_{tr}}\bigl(|D_{tr}|^{-s} [D_{tr},\pi_{tr}(f)]^*[D_{tr},\pi_{tr}(g)]\bigr)}{\zeta_{tr}(s)}\,.
\end{multline}
We emphasize that this expression has yet little meaning, as we have not yet specified a domain for this form. 
For example, while strongly pattern equivariant functions \cite{Kellendonk-PatEquiv} are dense they do not form an interesting core, as $Q_{tr}$ vanishes on such functions (see the paragraph on the transversal form in Section~\ref{ssec-DirForm}).

\subsection{The longitudinal triple for a substitution tiling}
\label{ssec-lgST}
We now aim at constructing what we call the longitudinal spectral triple for the substitution tiling which is based on the reversed substitution graph $\tilde\Gg=(\Vv,\tilde\Ee)$ ($\Gg$ with all orientations of the edges reversed, so with adjacency matrix $\tilde{A}=A^T$).
Set ${\tilde \varepsilon}^* = \varepsilon^*$ and choose $\tilde\tau$.
We take $\rho=\rlg$ as a parameter, and choose a subset
\begin{eqnarray*}
\hat \Hh_{lg}\subset \Hh(\tilde\Gg) & = & \left\{ (\tilde \varepsilon,\tilde \varepsilon') \in \tilde\Ee\times\tilde\Ee \, : \, \tilde \varepsilon\neq \tilde \varepsilon', \; s(\tilde \varepsilon)=s(\tilde \varepsilon') \right\} \\
& = & \left\{ (\varepsilon,\varepsilon') \in \Ee\times\Ee \, : \, \varepsilon\neq \varepsilon', \; r(\varepsilon')=r(\varepsilon') \right\}\,
\end{eqnarray*}
again satisfying the condition of  Lemma~\ref{lem-Gtau}. 
We denote the horizontal edges of level $n\in\NM$ by
\[
\Hh_{{lg},n}= \Bigr\{ (\eta \tilde \varepsilon,\eta \tilde \varepsilon') \, : \, \eta \in \Pi_{n-1}(\tilde\Gg), (\tilde \varepsilon,\tilde \varepsilon')\in \hat \Hh_{lg}
\Bigr\} \subset \Pi_n(\tilde\Gg) \times \Pi_n(\tilde\Gg) \,,
\]
and define the longitudinal approximation graph \(G_{lg,\tau}=(V_{lg},E_{lg})\) as in Section~\ref{sec-ST-Bratteli} by
\begin{eqnarray*}
V_{lg} = \bigcup_n V_{lg,n}\,, & V_{lg,n} = \tau (\Pi_n(\tilde\Gg)) \subset \Pi_\infty^\ast(\tilde\Gg)\,,\\
E_{lg} = \bigcup_n E_{lg,n}\,, & E_{lg,n} = \tau\times\tau(\Hh_{{lg},n})\,,
\end{eqnarray*}
together with the orientation inherited from $\hat \Hh_{lg}$: so \(E_{lg,n} = E_{lg,n}^{+} \cup E_{lg,n}^{-}\) for all $n\in\NM$, and \(E_{lg}=E_{lg}^{+}\cup E_{lg}^-\).

With these choices made, Theorem~\ref{thm-ST} provides us with a spectral triple for the algebra $C(\Pi_\infty(\tilde\Gg))$.

A longitudinal horizontal edge $h\in\hat \Hh_{lg}$ 
has the following interpretation:
As for the transversal horizontal edges, $\tau(s(h))$ and $\tau(r(h))$ start on a common vertex $v_0$, differ on their first edge and then come back to finish equally. To obtain their interpretation it is more useful, however, to reverse their orientation as this is the way the Robinson map $\Rob$ was defined. Then $h=(\tilde\varepsilon,\tilde\varepsilon')$ with $r(\tilde\varepsilon) = r(\tilde\varepsilon')$ determines a pair of microtiles $(t,t')$ of type $s(\tilde\varepsilon)$ and $s(\tilde\varepsilon')$, respectively, in a tile of type $r(\varepsilon)$.
The remaining part of the double path  $(\tau(\tilde \varepsilon),\tau(\tilde \varepsilon'))$ serves to fix a point in the two microtiles. Of importance is now the vector of translation $a_h$ between the two points of the microtiles. 

Similarily, an edge in $E_{lg,n}$ will describe a pair of $(n+1)$-th order microtiles in an $n$-th order microtile. 
By selfsimilarity again, the corresponding translation vector $a_e\in \RM^d$ between the two $(n+1)$-th order microtiles will satisfy
\begin{equation}
\label{eq-translg}
a_e= \theta^{-n} a_h\,.
\end{equation}
if $e\in E_{lg,n}(h)$.
See Figure~\ref{fig-chair4} for an illustration.
\begin{figure}[htp]
\begin{center}
\includegraphics[scale=0.45]{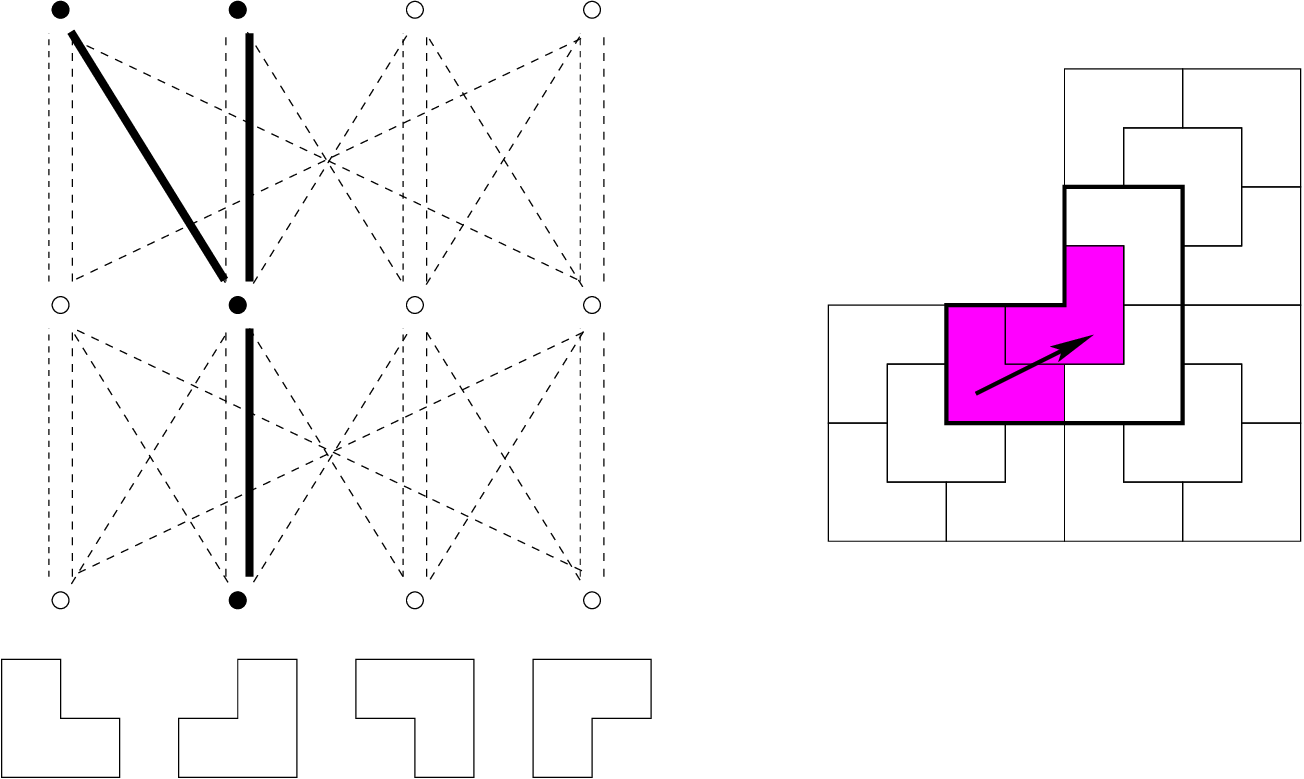}
\caption{{\small A microtile pattern associated with a horizontal arrow $e\in E_{lg,2}(h)$ (the pattern shown has the size of a single tile).
The arrow represents the vector $r_e$.}} 
\label{fig-chair4}
\end{center}
\end{figure}

Remember from Remark~\ref{rem-Robinson}, that we can identify $\Pi_\infty(\tilde\Gg) = \Pi_{-\infty,0}(\Gg)$.
And the inverse of the Robinson map $\Rob$ also induces a dense map $\Pi_{-\infty,0}^v \to t_v$ which is one-to-one on the pre-image of $\Pi_{-\infty,0}^*(\Gg)$; we still denote this map by $\Rob^{-1}$.
Hence the approximation graph for $\Pi_{-\infty,0}^v$ is also an approximation graph for $t_v$. 
Let $E_{lg}^v$ denote the set of edges whose corresponding paths pass through $v$ at level $0$.
We may thus adapt the above spectral triple to get the spectral triple $(C(t_v),\HS_{lg}^v,D_{lg})$ (which depends on $\rlg$) with representation and Dirac defined as in equations~\eqref{eq-rep} and~\eqref{eq-Dirac} by:
\[
\HS_{lg}^v=\ell^2(E_{lg}^v) \,,\quad 
\pi_{lg}(f)\phi(e) = f\bigl( \Rob^{-1}(s(e)) \bigr)\phi(e)\,, \quad 
D_{lg} \,\phi(e) = \rlg^{-n} \phi(e^\op)\,, \  e\in E_{lg,n}^v\,.
\]
That the bounded commutator axiom is satisfied follows from the following Lemma and the fact that H\"older continuous functions are dense in $C(t_v)$.
\begin{lemma}
If $f\in C(t_v)$ is H\"older continuous (w.r.t.\ the Euclidean metric $d$) with exponent $\alpha = \frac{-\log(\rlg)}{\log(\theta)}$, then $[D_{lg},\pi_{lg}(f)]$ is bounded. 
\end{lemma}
\begin{proof}
Suppose $f\in C(t_v)$ is H\"older continuous with exponent $\alpha = \frac{-\log(\rho_{lg})}{\log(\theta)}$, that is, $\left|\frac{f(x)-f(y)}{d(x,y)}\right|\leq C$ for some $C>0$ and all $x,y\in t_v$ then  
\[
\|[D,\pi(f)]\| = \sup_n\sup_{e\in E_{lg,n}} \left|\frac{f(\Rob^{-1} (r(e)))-f(\Rob^{-1}(s(e)))}{d(\Rob^{-1} (r(e)),\Rob^{-1}(s(e)))^\alpha}\right| \frac{d(\Rob^{-1} (r(e)),\Rob^{-1}(s(e)))^\alpha}{\rlg^n}. 
\]
this expression is finite since the first factor is bounded by $C$.
By self-similarity there exists $C'>0$ such that $d(\Rob^{-1} (r(e)),\Rob^{-1}(s(e)))\leq C'\theta^{-n}$. 
And $\alpha$ has been chosen so that $\theta^{-n\alpha}\rlg^{-n} = 1$.
\end{proof}
We refer to this spectral triple $(C(t_v),\HS_{lg}^v,D_{lg})$ as the {\em longitudinal spectral triple for the prototile $t_v$}. 
It should be noted that, although the map $\Rob^{-1}$ is continuous, the topology of $t_v$ and $\Pi_{-\infty,0}^v$ are quite different and so the Connes distance of this spectral triple does not induce the topology of $t_v$.
By Theorems~\ref{thm-zeta} and~\ref{thm-specmeas} the longitudinal spectral triple has metric dimension $s_{lg}=\frac{ d\log(\theta)}{-\log(\rlg)}$ for all $v$, but what depends on $v$ is the residue of the zeta function.
In fact, as compared to the zeta function of the full triple it has to be rescaled: \(\zeta_{tr}^v= (L_v/\sum_u L_u) \zeta_{tr}\).

The spectral measure $\mu_{lg}^v$ is easily seen to be the normalized Lebesgue measure on $t_v$, as the groupoid of tail equivalence acts by partial translations. 

Similarily to the transversal case we are interested in the quadratic form defined formally by 
\begin{multline}
Q_{lg}^v(f,g) = \Tt_{lg}^v([D_{lg},\pi_{lg}(f)]^*[D_{lg},\pi_{lg}(g)]) \\= 
\lim_{s\downarrow s_{lg}} 
\frac{\TR_{\HS_{lg}^v}\bigl(|D_{lg}|^{-s} [D_{lg},\pi_{lg}(f)]^*[D_{lg},\pi_{lg}(g)]\bigr)}{\zeta_{lg}^v(s)}\,.
\end{multline}
Again, this expression has yet little meaning, as we have not yet specified a domain for this form. 

\subsection{The spectral triple for $\Omega_\Phi$}
\label{ssec-SThull}
We now combine the above triples to get a spectral triple
\((C(\Omega_\Phi),\HS,D)\) for the whole tiling space
$\Omega_\Phi$.  
The graphs $\Gg$ and $\tilde\Gg$ have the same set of vertices $\Vv$,
so we notice from Remark~\ref{rem-Robinson} that the identification  
\[
\Pi_{-\infty,+\infty}(\Gg) 
= \bigcup_{v\in\Vv} \Pi_{-\infty,0}^v(\Gg) \times \Pi_{0,+\infty}^v(\Gg)  
= \bigcup_{v\in\Vv}   \Pi_{\infty}^v(\tilde\Gg) \times \Pi_{\infty}^v(\Gg)
\]
suggests to construct the triple for $\Omega_\Phi$ as a direct sum
of tensor product spectral triples related to the transversal and the
longitudinal parts.
In fact, $\Pi_{-\infty,0}^v(\Gg) \times \Pi_{0,+\infty}^v(\Gg)$ is dense in $t_v\times \Xi_{t_v}$ (see Remark~\ref{rem-Robinson}) and so we can use the tensor product construction for spectral triples to obtain a spectral triple for $C(t_v\times \Xi_{t_v})\cong C(t_v)\otimes C(\Xi_{t_v})$ from the two spectral triples considered above.
Furthermore, the $C^\ast$-algebra $C(\Omega_\Phi)$ is a subalgebra of $\bigoplus_{v\in\Vv} C(t_v\times \Xi_{t_v})$ and so the direct sum of the tensor product spectral triples for the different tiles $t_v$ provides us with a spectral triple for $C(\Omega)$:
\begin{equation}
 \label{eq-STOmega}
\HS=\bigoplus_{v\in\Vv} \HS_{tr}^v \otimes \HS_{lg}^v\,, \quad
\pi=\bigoplus_{v\in\Vv} \pi_{tr}^v\otimes \pi_{lg}^v\,, 
\quad D= \bigoplus_{v\in\Vv} \bigl( D_{tr}^v\otimes \id + \chi \otimes D_{lg}^v  \bigr)\,,
\end{equation}
where $\chi$ is the grading of the transversal triple (which flips the
orientations in $E_{tr}$). 
The representation of a function $f\in C(\Omega)$ then reads
\begin{equation}
 \label{eq-COmega}
 \pi(f) = \sum_{v\in\Vv} f_{tr}^v \otimes f_{lg}^v\,, \qquad \text{\rm with} \quad f_{tr}^v = \pi_{tr}^v(f) \in C(\Xi_{t_v})\,, \ f_{lg}^v = \pi_{lg}^v(f) \in C(t_v)\,.
\end{equation}

From the results in Section~\ref{ssec-tensorprod} we now get all the
spectral information of \((C(\Omega_\Phi),\HS,D)\). To formulate
our results more concisely let us call 
$A\in \Bb(\HS)$ {\em non resonant}, if \(\bar{A}_n
\stackrel{n\rightarrow \infty}{\sim} c_A e^{\imath n \phi}\) for some
$c_A>0$ and non resonant $\phi\in (0,2\pi)$ (Definition~\ref{def-resphase}). 
\begin{theo}
\label{thm-STOmega}
The above is a spectral triple for $C(\Omega_\Phi)$.
Its spectral dimension is
\[
s_0= s_{tr}+s_{lg} = \frac{d\log\theta}{-\log\rtr} +
\frac{d\log\theta}{-\log\rlg}\,, 
\]
and its zeta function $\zeta(z)$ has a simple pole at $s_0$ with strictly positive residue.
Moreover, for $A = \oplus_v A_{tr}^v\otimes A_{lg}^v$, 
with either both $A_{tr}^v$ and $A_{lg}^v$ strongly regular, or one is strongly regular and the other the sum of a strongly regular and a non resonant part, then we have 
\begin{equation}
\label{eq-decspecstate}
\Tt(A) = \sum_{v\in\Vv}  \mbox{\rm freq}(t_v) \mbox{\rm vol}(t_v) \;
\Tt_{tr}^v(A_{tr}^v) \Tt_{lg}^v(A_{lg}^v)  \,.
\end{equation}
In particular, the spectral measure is the unique invariant ergodic probability measure $\mu$ on $\Omega_\Phi$. 
\end{theo}
\begin{proof} Consider first the triple for the matchbox $t_v\times
  \Xi_{t_v}$, that is, the tensor product spectral triple for
  $C(t_v)\otimes C(\Xi_{t_v})$. Applying Lemma~\ref{lem-prod} we
  obtain the value $s_0^v = \frac{d\log\theta}{-\log\rtr} +
  \frac{d\log\theta}{-\log\rlg}$ for the abscissa of convergence of
  its zeta function $\zeta^v$. In particular, this value does not
  depend on $v$.
Furthermore,  
\[
\lim_{s\to s_0^+}(s-s_0)\zeta^v(s) =  
\frac{\freq(t_v)\vol(t_v)}{\sum_u \vol(t_u)} \frac{\sum_{k=-\infty}^\infty  
\Gamma(\frac{d \log \theta +2\pi i k}{-2\log (\rho_{tr})})  \Gamma(\frac{d \log \theta -2\pi i k}{-2\log (\rho_{lg})})  }
{2\Gamma(\frac{s_0}2) \log(\rho_{tr}) \log(\rho_{lg})}. 
\]
The above number is in fact a strictly positive real number, as it is up to a positive factor the mean of two strictly positive periodic functions.
It follows that the abscissa of convergence for the zeta function of
the direct sum of the above triples $\zeta$ is equal to the common
value $s_0=s_0^v$.
From this we can now determine with the help of Lemma~\ref{lem-prodstate} and (\ref{eq-sum-state}) the spectral state.

If $A_{tr}^v$ and $A_{lg}^v$ are both strongly regular then, by Corollary~\ref{cor-prodstrongreg}, \(\Tt(A_{tr}^v\otimes A_{lg}^v) =n_v\Tt_{tr}^v(A_{tr}^v)\Tt_{lg}^v(A_{lg}^v)\), with the factor \(n_v=\freq(t_v)\vol(t_v)\) because the states are normalized. 
If, say, $A_{tr}^v$ is regular and $A_{lg}^v=A_{lg,reg}^v+A_{lg,nres}^v$ is the sum of a strongly regular and a non resonant part, then 
\begin{eqnarray*}
\Tt(A_{tr}^v\otimes A_{lg}^v) & = & \Tt(A_{tr}^v \otimes A_{lg,sreg}^v) + \Tt(A_{tr}^v \otimes A_{lg,nres}^v) \\
& = & n_v\Tt_{tr}^v(A_{tr}^v)\Tt_{lg}^v(A_{lg,sreg}^v) + \freq(t_v)\Tt_{tr}^v(A_{tr}^v) \Tt(\id \otimes A_{lg,nres}^v) \\
& = & n_v \Tt_{tr}^v(A_{tr}^v)\Tt_2(A_{lg,sreg}^v) + 0\\
& = & n_v\Tt_{tr}^v(A_{tr}^v)\Tt_{lg}^v(A_{lg,sreg}^v) + n_v\Tt_{tr}^v(A_{tr}^v)\Tt_{lg}^v(A_{lg,nres}^v) \\
& = & n_v\Tt_{tr}^v(A_{tr}^v)\Tt_{lg}^v(A_{lg}^v)\,,
\end{eqnarray*}
where the second line follows by Corollary~\ref{cor-prodstrongreg}, the third by Lemma~\ref{lem-resphi}, and the forth by Lemma~\ref{lem-stateres} (the state of a non resonant operator vanishes: \(\Tt_{lg}^v(A_{lg,nres})=0\)).
The argument is the same if $A_{lg}^v$ is strongly regular, and $A_{tr}^v$ the sum of a strongly regular and a non resonant part.
Hence we get in both cases
\[
\Tt(A) = \Tt\bigl( \sum_v A_{tr}^v \otimes A_{lg}^v \Bigr) = \sum_v n_v \Tt_{tr}^v(A_{tr}^v)\Tt_{lg}^v(A_{lg}^v) \,.
\]
Since $n_v=\freq(t_v)\vol(t_v)$ is the $\mu$-measure of the matchbox $t_v\times \Xi_{t_v}$ we see that the spectral measure coincides with $\mu$.
\end{proof}

\subsection{Pisot substitutions}
\label{ssect-Pisot}
Recall that we consider here aperiodic primitive FLC substitutions which admit a fixed point tiling.
Their dilation factor $\theta$ is necessarily an algebraic integer.
There is a dichotomie: either the dynamical system defined by the tilings is weakly mixing (which means that there are no non-trivial eigenvalues of the translation action)
or the dilation factor $\theta$ of the substitution is a Pisot number, that is, an algebraic integer greater than $1$ all of whose Galois conjugates have modulus strictly smaller than $1$. 
In the second case the substitution is called a Pisot-substitution. Let us 
recall here some of the relevant results. 
We denote by $\|x\|$ the distance of $x\in\RM$ to the closest integer. 
A proof of the following theorem can be found in
\cite{Cassels}[Chap.~VIII, Thm.~I].  
\begin{theo}[Pisot, Vijayaraghavan]\label{thm-Pisot}
Let $\theta>1$ be a real algebraic number and $\alpha\neq 0$ another real such that
$\|\alpha\theta^n\|\stackrel{n\to\infty}{\longrightarrow} 0$. Then we have
\begin{itemize}
\item
$\theta$ is a Pisot number. We denote $\{\theta_j:1\leq j\leq J\}$ the conjugates of $\theta$ with $\theta_1=\theta$ and $|\theta_{j+1}|\leq|\theta_j|$. 
\item $\alpha\in\QM[\theta]$, i.e.\ $\alpha=p_\alpha(\theta)$ for some polynomial $p_\alpha$ with rational coefficients.
\item There exists $n_0$ such that for $n\geq n_0$
$$\sum_{j=1}^Jp_\alpha(\theta_j)\theta_j^n\in \ZM.$$
If $\theta$ is unimodular then $n_0=0$.
\end{itemize}
\end{theo}
We note that $p_\alpha(\theta_j)\neq 0$ for all $j$, since otherwise $p_\alpha$ would be divisible by the minimal polynomial of $\theta$ implying that also $p_\alpha(\theta)= 0$.

We assume throughout this work that $\theta$ is irrational so that there is at least one other conjugate ($J>1$).
Note that $\theta^{-1}$ is a polynomial (of degree $J-1$) in $\theta$
with coefficients in $\frac{1}{N}\ZM$ where $N$ is the constant term
in the minimal polynomial for $\theta$. $\theta$ is unimodular precisely if $N = \pm 1$. 

Recall that a dynamical eigenfunction to eigenvalue $\beta\in{\R^d}^*$, is a measurable function $f_\beta:\Omega_\Phi\to\CM$ which satisfies \(f_\beta(\omega+t)=e^{2\pi \beta(t)} f_\beta(\omega)\) for almost all $\omega\in \Omega_\Phi$ (w.r.t.\ the ergodic measure $\mu$) and all $t\in \RM^d$. If $f_\beta$ can be chosen continuous then $\beta$ is also called a continuous eigenvalue.
The set of eigenvalues forms a group which we denote $E$.
We call a vector $r\in\R^d$ a return vector (to tiles) if it is a vector between the punctures of two tiles of the same type in some tiling $T\in\Omega_\Phi$. 
\begin{theo}[\cite{Sol07}]
Consider a substitution with dilation factor $\theta$. The following are equivalent.
\begin{enumerate}
\item $\beta$ is an eigenvalue of the translation action.
\item $\beta$ is an eigenvalue of the translation action with continuous eigenfunction.
\item For all return vectors $r$ one has $\|\beta(r)  \theta^n\|\longrightarrow 0$. 
\end{enumerate} \end{theo}
In particular, we may assume that all eigenfunctions are continuous, and if there are non-zero eigenvalues then $\theta$ must be a Pisot number and for all return vectors (to tiles) $r$ and large enough $n$ we have
\begin{equation}
\label{eq-Pisot}
\exp(2\pi\imath \beta(r)\theta^n) =  \exp\Bigl(-2\pi\imath \sum_{j=2}^Jp_{\beta(r)}(\theta_j)\theta_j^n \Bigr)
\end{equation}
We would like to give a geometric interpretation of the values of
$p_{\beta(r)}(\theta_j)$. To better illustrate this we consider first only the unimodular situation and
explain what changes in the general case later.
The following theorem can be found in \cite{BK}.
\begin{theo}\label{thm-BK}
Consider a $d$-dimensional substitution as above with dilation factor $\theta$. 
If $\theta$ is a unimodular Pisot number of degree $J$, then the group $E$ of eigenvalues is a dense subgroup of ${\RM^d}^*$ of rank $dJ$.
\end{theo}
Recall for instance from \cite{BK} that the maximal equicontinuous factor of 
$(\Omega_\Phi,\RM^d)$ can be identified with $\hat E$ the Pontrayagin dual of the group of 
eigenvalues $E$ equipped with the discrete topology. Its $\RM^d$ action is 
induced by translation of eigenfunctions, namely $\alpha_r:\hat E\to \hat E$, $r\in\RM^d$, acts as
$(\alpha_r(\chi))(\beta) = e^{2\pi\imath \beta(r)} \chi(\beta)$, $\chi\in\hat E$, $\beta\in E$.
The factor map $\pi:\Omega\to \hat E$ is dual to the embedding of the eigenfunctions in $C(\Omega_\Phi)$. We may choose an element $T^*\in\Omega_\Phi$ and consider $\pi(T^*)$
as the neutral element in $\hat E$. 

$E$ is a free abelian group of rank $dJ$. We can therefore identify it with a regular lattice in an 
$dJ$-dimensional Euclidean space $\Ue$ equipped with a scalar product $\langle\cdot,\cdot\rangle$, and $\hat E$ with $\Ue/E^{rec}$ via the map 
$\Ue/E^{rec}\ni \xi \mapsto e^{2\pi\imath \langle\xi,\cdot\rangle}\in \hat E$ ($E^{rec}$ is the reciprocal lattice).
With this description of $\hat E$ the dual pairing $\hat E\times E \to \CM$ becomes
$\Ue/E^{rec}\times E \ni (\xi,\beta) \mapsto e^{2\pi\imath \langle\xi,\beta\rangle}\in \CM$. Furthermore 
$e^{2\pi\imath \langle\cdot,\beta\rangle}$, $\beta\in E$, is an eigenfunction of
the $\R^{dJ}$-action by translation on $\Ue/E^{rec}$ to the eigenvalue $\langle\cdot,\beta\rangle\in
{\RM^{dJ}}^*$. 
Finally, the action of $\alpha_r$ becomes  $(\alpha_r(\xi))(\beta) = e^{2\pi\imath \beta(r)} \xi(\beta)$ and hence
$\alpha_r(0)$ is the unique element of $ \Ue/E^{rec}$ satisfying
$\beta(r) - \langle\alpha_r(0),\beta\rangle\in\Z$ for all $\beta\in E$.
Now $r\mapsto\alpha_r(0)$ 
is continuous and locally free. 
It follows that the map $r\mapsto \alpha_r(0)$ lifts to a linear
embedding of $\R^d$ into the universal cover $\Ue$ of $\hat
E$. We denote the lift of $\alpha_r(0)$ with $\tilde r$. The vector $\tilde r$ is thus defined by
\[ 
\beta(r) = \langle\tilde r,\beta\rangle,\quad  \forall \beta\in E\,.
\]
The image of the embedding, which we denote by $U$, is simply
the lift of the orbit of $\pi(T^*)$. The acting group $\RM^d$, or equivalently the orbit of $\pi(T^*)$,
can therefore be identified with a subspace $U$ of a euclidean space $\Ue$. Our next aim is to construct a complimentary subspace $S$ to $U$.

The endomorphism on ${\RM^d}^*$ which is dual to the
linear endomorphism $\theta\id$ on $\RM^d$ preserves the group of eigenvalues $E\subset {\RM^d}^*$. It therefore restricts to a group endomorphism of $E$.
We denote this restriction by $\varphi$. 
With respect to a basis of $E$, $\varphi$ is thus an integer $dJ\times dJ$ matrix. Denote by $\varphi_\RM$ its linear extension to $\Ue$ and by $\varphi^t_\RM$ the transpose\footnote{Although taking the transpose is a dualization we have not returned to the original map $\theta\id$ since 
the two dualizations are w.r.t.\ two different dual pairings.}
of $\varphi_\RM$. Then we have
$$\langle\varphi^t_\RM(\tilde r),\beta\rangle=\langle\tilde r,\varphi_\RM(\beta)\rangle=\phi(\beta)(r) = \theta\beta(r) = \theta\langle\tilde r,\beta\rangle$$
showing that $\varphi^t_\RM(\tilde r)=\theta \tilde r$.
We know also from \cite{BK} that $\varphi_\RM$ has eigenvalues $\theta_j$, $1\leq j\leq J$ each with multiplicity $J$. It follows that $U$ is the eigenspace of $\varphi^t_\RM$ to the eigenvalue $\theta$ 
and so by the Pisot-property is the full unstable subspace of $\varphi^t_\RM$. We let $S$ be the stable subspace of  $\varphi^t_\RM$.

\paragraph{Example:} Before we continue the description we provide the example of the Fibonacci substitution tiling. Here $\theta=\frac{1+\sqrt{5}}2$ is the golden mean and $E$ is the rank $2$ subgroup $\Z[\theta]\subset\RM^*$. We choose the basis $\{\theta,1\}$ for $E$. Then $\varphi$ has matrix elements
$\left(\begin{array}{cc}1 &1\\ 1& 0\end{array}\right)$. $\varphi_\RM$ has, of course, the same matrix expression. It follows that $U$ is the subspace of $W=\RM^2$ generated by the vector $(\theta,1)$ and $S$ is the subspace generated by $(-\theta^{-1},1)$. It is no coincidence that this is the cut \& project setup for the Fibonacci tiling, except that we do not have a window.
 
 \medskip
 
We return to the geometric description of the polynomicals.
Let $r$ be a return vector. 
We choose a lattice base $\{\beta_{\nu}\}_{\nu=1\cdots,dJ}$ for $E$
and let $\{\beta^{\nu}\}_{\nu=1\cdots,dJ}$ be the dual base for $E^{rec}$.
We can express $\tilde r$ in the dual base. Since $(\tilde r,\beta_\nu) =\beta_\nu(r)$ we have
\[  
\tilde r
 = \sum_{\nu=1}^{dJ}   \beta_\nu(r) \beta^\nu\,.
 \]
By Theorem~\ref{thm-Pisot}
there exist polynomials $p_\nu$ with rational coefficients such that 
\[
 \beta_\nu(r) = p_\nu(\theta).
\]
In particular, the vector with coefficients $(p_\nu(\theta))$ is an eigenvector of $\varphi^t_\R$ to eigenvalue $\theta$.
The eigenvalue equation
$\sum_\mu (\varphi^t_{\nu\mu}-\theta\delta_{\nu\mu})p_\mu(\theta)$ is a set of $dJ$ polynomial equations with rational coefficients and hence it is satisfied for $\theta$ whenever it is satisfied for all conjugates $\theta_j$.
Thus  $(p_\nu(\theta_j))$ is an eigenvector of $\varphi^t_\R$ to eigenvalue $\theta_j$.
If $j>1$ this vector lies thus in the stable manifold $S$ of $\varphi^t_\R$ at $0$.
We may thus define the following {\em star map} 
${}^*:\{\tilde r
: r\mbox{ is a return vector}\}\subset U\to S$, 
\begin{equation}\label{eq-star}
r\mapsto \tilde r
^* = \sum_{j=2}^J\sum_{\nu=1}^{dJ}  p_\nu(\theta_j) \beta^\nu\,.
\end{equation}
The map $\tilde r\mapsto \tilde r^*$ is actually Moody's star map. 
Indeed, $W$ decomposes into the direct sum of vector spaces  $W= S\oplus U$ and contains
$E^{rec}$ as regular lattice. Let $\pi_U$ and $\pi_S$ be the projection onto $U$ and $S$, resp., with kernel $S$ and $U$, resp.  
Recall that  $\sum_{j=1}^J p_\nu(\theta_j) \in \Z$ (we assumed that $\theta$ is unimodular).
This can be reinterpreted as $\tilde r + \tilde r^*\in E^{rec}$.
Since $S$ intersects $E^{rec}$ only in the origin\footnote{$\varphi^t$ preserves the intersection $E^{rec}\cap S$ which is hence invariant under a strictly contracting map, and since the intersection is uniformly discrete it must be $\{0\}$.}, 
$\tilde r^*$ is the unique vector $w\in S$ which satisfies
$\tilde r + w \in E^{rec}$. Stated differently,  
$\tilde r^*=\pi_S \circ \pi_U^{-1}(\tilde r)$, is the image under $\pi_S$ 
of a preimage in $E^{rec}$ of $\tilde r$ under $\pi_U$.

Let $S_2$ be the subspace of $S$ generated by the subleading conjugates of $\theta$, i.e.\ the eigenspace of $\varphi^t_\RM$ to the eigenvalues $\theta_2,\cdots,\theta_L$ which are characterised by the property that they all have the same modulus: $|\theta_j|=|\theta_2|$ for all $2\leq j\leq L\leq J$. We define the {\em reduced star map} ${}^\star:U\to S_2$ by  
\begin{equation}\label{eq-redstar}
r\mapsto \tilde r
^\star = \sum_{j=2}^L\sum_{\nu=1}^{dJ}  p_\nu(\theta_j) \beta^\nu\,.
\end{equation}
By linearity we get 
\begin{equation}\label{eq-geom}
\sum_{j=2}^L p_{\beta(r)}(\theta_j)  = \langle\tilde r^\star,\beta\rangle
\end{equation}
which yields the geometric interpretation of the polynomials $p_{\beta(r)}$ we looked for.

\begin{rem}{\rm
If $\theta$ is not unimodular the arguments are principally the same except for having to work with 
inverse limits. The basic differences are that $n_0$ might be strictly larger than $0$ in Theorem~\ref{thm-Pisot} and that $\varphi$ is no longer invertible. In fact, in the non uni-modular case 
Theorem~\ref{thm-BK} has to be modified in the following way \cite{BK}:
there exists a dense rank $dJ$ subgroup $F$ such that $E = \lim_\to (F,\varphi)$. In particular,
$\hat E$ is an inverse limit of $dJ$-tori. Now one can construct the euclidean space $\Ue$ with its stable and unstable subspaces under $\varphi^t_\RM$ as above but for $F$ instead of $E$. 
Then $\hat E$ corresponds to the subgroup $E^{rec}:=\bigcup_{n\geq 0}
{\varphi^t_\RM}^{-n}( F^{rec})\subset \Ue$.\footnote{Note that
  $E^{rec}$ is defined by this union; it is not the reciprocal lattice of $E$ which wouldn't make sense, as $E$ is not a regular lattice in $\Ue$.}
With these re-interpretations, the map $r\mapsto \tilde r^*$ is the same, namely $\tilde r$ is the lift of $\alpha_r(0)$, where $\alpha_r$ the action of $r\in\R^d$ on the torus $\hat F$, and $\tilde r^*$ is the unique vector $w\in S$ which satisfies $\tilde r + w \in E^{rec}$. }
\end{rem}

\subsection{Dirichlet forms}
\label{ssec-DirForm}

The goal of this section is to investigate when the formal expression
for the Dirichlet forms can be made rigourous. 
We will see that, apart from trivial cases, this fixes the values for
$\rho_{tr}$ and $\rho_{lg}$. 
The study is technical.
We explain the steps of the derivation which are similar for both
forms, but we do not give all the straightforward but tedious details
of the technical estimates.

We start by specializing the formal expression of the Dirichlet form 
given in equations~\eqref{eq-Q2} and~\eqref{eq-qn} to our set-up.
According to equation~\eqref{eq-COmega}, we view the representation of
$h\in C(\Omega_\Phi)$ as the sum of elementary functions \( \pi(h)=
\sum_{v\in\Vv} h_{tr}^v \otimes h_{lg}^v\). 
And by Theorem~\ref{thm-specmeas}, any such $h_{\alpha}^v$ is a strongly regular operator on $\HS_{\alpha}^v$.
Hence, by Lemma~\ref{lem-prodform}, we can decompose the form as follows
\[
Q(f,g) = \Tt\bigl( [D,\pi(f)]^\ast [D,\pi(g)] \bigr) =
Q_{lg}(f,g) + Q_{tr}(f,g)\,,
\]
with
\begin{subequations}
\label{eq-DirForm-lgtr}
\begin{equation}
\label{eq-DirForm-lg}
Q_{lg}(f,g) = \sum_{v\in\Vv} \freq(t_v) \,
\Tt\bigl( \id\otimes [D_{lg}^v,f_{lg}^v]^\ast [D_{lg}^v,g_{lg}^v]  \bigr)
\Tt_{tr}^v\bigl (f_{tr}^{v\ast} g_{tr}^v) \bigr)
\end{equation}
\begin{equation}
\label{eq-DirForm-tr}
Q_{tr}(f,g) = \sum_{v\in\Vv} \vol(t_v) \, 
\Tt\bigl( [D_{tr}^v,f_{tr}^v]^\ast [D_{tr}^v,g_{tr}^v] \otimes \id \bigr)
\Tt_{lg}^v \bigl( f_{lg}^{v\ast} g_{lg}^v \bigr)
\end{equation}
\end{subequations}
We are going to show in the next two paragraphs, that for $f,g$ in a suitable core, the operators \([D_{lg}^v,f_{lg}^v]^\ast [D_{lg}^v,g_{lg}^v]\) are strongly regular, and the operators \([D_{tr}^v,f_{tr}^v]^\ast [D_{tr}^v,g_{tr}^v]\) are the sums of a strongly regular and a non resonant part.
Hence by Corollary~\ref{cor-prodstrongreg} and Lemma~\ref{lem-resphi} we get the following decomposition of the forms:
\begin{subequations}
\label{eq-DirForm-lgtr2}
\begin{equation}
\label{eq-DirForm-lg2}
Q_{lg}(f,g) = \sum_{v\in\Vv} n_v \,
Q_{lg}^v(f_{lg}^v,g_{lg}^v) \int_{\Xi_{t_v}} f_{tr}^{v\ast} g_{tr}^v \,d \mu_{tr}^v \,,
\end{equation}
\begin{equation}
\label{eq-DirForm-tr2}
Q_{tr}(f,g) = \sum_{v\in\Vv} n_v \, Q_{tr}^v(f_{tr}^v,g_{tr}^v)  \int_{t_v} f_{lg}^{v\ast} g_{lg}^v \,d \mu_{lg}^v \,,
\end{equation}
\end{subequations}
where \(n_v= \vol(t_v)\freq(t_v)\) is the normalization factor in
equation~\eqref{eq-decspecstate}, $d\mu_{tr}^v$ is the normalized
invariant measure on $\Xi_{t_v}$ (by the results of
Section~\ref{ssec-trST}), and $d \mu_{lg}^v$ the normalized Lebesgue
measure on $t_v$ (by the results of Section~\ref{ssec-lgST}). 

Each of the forms $Q_{lg}^v$ and $Q_{tr}^v$ are of the type given in
Section~\ref{ssec-form}, we only have to substitute the
expression
\begin{equation}
\label{eq-deltaef}
\delta_e^\alpha f = \frac{f\circ\Rob^{-1}(r(e)) -f\circ
  \Rob^{-1}(s(e))}{\rho_\alpha^n} \,. 
\end{equation}
in $q_{\alpha,n}^v$ in equation~\eqref{eq-qn} for $Q_{\alpha}^v$, $\alpha\in\{tr,lg\}$ and
$v\in\Vv$.
It follows from Proposition~\ref{prop-form}
that the forms $Q_{tr}^v$ and $Q_{lg}^v$ are symmetric, positive definite, and
Markovian on the domains 
\(
\Dd_{tr}^v=\bigl\{f \in L^2_{\RM}(\Xi_{t_v},d \mu_{tr}^v) \, :
\, Q_{tr}^v(f,f) < +\infty \bigr\}
\) and \(
\Dd_{lg}^v=\bigl\{f \in L^2_{\RM}({t_v},d \mu_{lg}^v) \, :
\, Q_{lg}^v(f,f) < +\infty \bigr\}
\), respectively.
\paragraph{The longitudinal form}
Let us first look at the longitudinal part $Q_{lg}^v$, which is simpler.
We show that $q^v_{\alpha,n}(f,g)$ has a limit for suitable $f,g$.
This will prove that \([D_{lg}^v,f_{lg}^v]^\ast [D_{lg}^v,g_{lg}^v]\) is strongly regular by Corollary~\ref{cor-cesar}, and will imply the decomposition of $Q_{lg}$ given in equation~\eqref{eq-DirForm-lg2}.

We wish to adapt the parameter $\rlg$ so as to obtain a non-trivial form with core which is dense in $C^1(t_v)$.

Given an edge $h\in\hat \Hh_{lg}$, let us denote by $a_h\in \RM^d$ the translation vector between the punctures of the microtiles associated with $s(h)$ and $r(h)$ in the decomposition of the tile associated with $s^2(h)=sr(h)$.
If $a_e \in \RM^d$ denotes the corresponding vector for $e\in E_{lg,n}^v(h)$,  equation~\eqref{eq-translg} gives $a_e=\theta^{-n}a_h$.
For $n$ large we can thus approximate (uniformly, by uniform continuity of $f$)
\[
 \dlg_{e} f \simeq \left( \frac{\theta^{-1}}{\rlg} \right)^{n} 
	\ (a_h \cdot \nabla) f(x_e)\,, \quad \text{\rm where} \quad x_e = \Rob^{-1}(s(e))\,.
\]
Choosing $f,g\in C^2_\RM(t_v)$, we can substitute the above approximation in $q_{lg,n}^v(f,g)$, up to an error term uniform in $e$, which gives 
\[
 q_{lg,n}^v(f,g) \simeq
\left( \frac{\theta^{-1}}{\rlg}\right)^{2n}  \frac{1}{\#E_{lg,n}^v}
\sum_{h\in \hat \Hh_{lg}}
\sum_{e \in E_{lg,n}^v(h)}  
(a_h \cdot \nabla) f(x_e) \; (a_h \cdot \nabla) g(x_e) \,,
\]
where $E_{lg,n}^v(h)$ is the set of edges  of type $h$ in $E_{lg,n}^v$.

In order to estimate the above sum, we decompose $t_v$ into boxes.
Assuming that $n$ is large, we choose an integer $l=l_{n}$ such that \(1<< l<<n\),  and consider the boxes \(B_\gamma=\Rr^{-1}([\gamma])\) for \(\gamma \in \Pi_{l}^v(\tilde\Gg)\), the set of paths of length $l$ which start at $v$ at level $0$ (here $[\gamma]$ is the set of infinite paths which agree with $\gamma$ from level $0$ to level $l$).
The idea is that $B_\gamma$ should be large enough to allow us to make averages over it, yet small enough for continuous functions to be approximately constant on it: \(F(x_e) \simeq F(x_\gamma)\), for some $x_\gamma\in B_\gamma$, and all $e$ for which $x_{e}\in B_\gamma$.
We have
\[
q_{lg,n}^v(f,g) \simeq  \left( \frac{\theta^{-1}}{\rlg}\right)^{2n}  \sum_{h\in \hat \Hh_{lg}}
\sum_{\gamma\in\Pi_{l}^v} 
\frac{\#E_{lg,n}^v(h,\gamma) }{\#E_{lg,n}^v}
\
(a_h \cdot \nabla_{lg}) f(x_{\gamma}) \; (a_h \cdot \nabla_{lg}) g(x_{\gamma}) \,,
\]
for some $x_\gamma\in B_\gamma$, and where $E_{lg,n}^v(h,\gamma)$ stands for the set of edges of type $h$ whose associated (range and source) infinite paths agree with $\gamma$ from level $0$ to level $l$.
Using the estimation of the powers of the substitution matrix in equation~\eqref{eq-An} (note that we use the {\em transpose} of $A$ for $\tilde\Gg$ here) we have
\[
\frac{\#E_{lg,n}^v(h,\gamma) }{\#E_{lg,n}^v} = \frac{L_{r(\gamma)} R_{s^2(h)}\pf^{n-1-l} (1+o(1))}{L_v\sum_{h'\in\hat \Hh_{lg}}R_{s^2(h')}\pf^{n-1} (1+o(1))}
= c_{lg} \,\freq(t_{s^2(h)}) \, \mu_{lg}^v(B_\gamma) \;(1+o(1))
\]
with
\begin{equation}\label{eq-clg}
c_{lg}= \Bigl( \sum_{h\in\hat \Hh_{lg}}R_{s^2(h)} \Bigr)^{-1}.
\end{equation}
The last step uses the approximation \(\mu_{lg}^v(B_\gamma)F(x_\gamma) \simeq \int_{B_\gamma} F(x) d^d \mu_{lg}^v\), to obtain
\[
 \left( \frac{\theta^{-1}}{\rlg}\right)^{-2n} q_{lg,n}^v(f,g)  \simeq q^v_{\alpha}(f,g) :=
c_{lg} \sum_{h\in\hat \Hh_{lg}} \freq(t_{s^2(h)})
\int_{t_v} 
(a_h \cdot \nabla) f \; (a_h \cdot \nabla) g \; d \mu_{lg}^v\,.
\]
Hence, if $\rlg=\theta^{-1}$, the sequence $q_{lg,n}^v(f,g)$ converges to $q_{lg}^v(f,g)$, therefore by Corollary~\ref{cor-cesar} we have \(Q_{lg}^v(f,g)=q_{lg}^v(f,g)\).
We now have to compute $Q_{lg}$ from equation~\eqref{eq-DirForm-lg2}, summing up the $Q_{lg}^v$ over $v\in\Vv$.
Notice that by the decomposition of the representation of a function as $\pi(h)=\sum_{v\in\Vv} h_{tr}^v\otimes h_{lg}^v$, we have
\[
\sum_{v\in\Vv} n_v \,
\int_{t_v} (a_h \cdot \nabla) f \; (a_h \cdot \nabla) g \; d \mu_{lg}^v
\int_{\Xi_{t_v}} f_{tr} g_{tr} d\mu_{tr}^v = 
 \int_{\Omega}  (a_h \cdot \nabla_{lg} )f  \; (a_h\cdot \nabla_{lg}) g \; d\mu\,,
\]
where $\nabla_{lg}$ is the {\em longitudinal gradient} on $\Omega_\Phi$: it takes derivatives along the leaves of the folliation; so it reads on the representation of $C(\Omega_\Phi)$ simply $\nabla_{lg}=\id\otimes \nabla_{\RM^d}$.
Define the operator on $L^2_{\RM}(\Omega_\Phi, d\mu)$: 
\begin{equation}
\label{eq-lgLaplace}
\Delta_{lg} = c_{lg} \nabla_{lg}^\dagger \Kk \nabla_{lg} \,, \quad \text{\rm with} \quad 
\Kk = \sum_{h\in\hat \Hh_{lg}} \freq(t_{s^2(h)}) a_h \otimes a_h\,.
\end{equation}
We thus we have 
\begin{equation}
Q_{lg} (f,f)
= \left\{ \begin{array}{ll}
 \langle f, \; \Delta_{lg} \; f\rangle_{ L^2_{\RM}(\Omega_\Phi, d\mu)} & \mbox{if } \rlg =\theta^{-1} \\ 
0 & \mbox{if } \rlg > \theta^{-1} \\
+\infty & \mbox{if } \rlg < \theta^{-1}
\end{array}\right. \,,
\quad \text{\rm for all } f\in C_{lg}^2(\Omega_\Phi)\,,
\end{equation}
where $C_{lg}^2(\Omega_\Phi)$ is the space of longitudinally $C^2$ functions on $\Omega_\Phi$.
So we see that for $\rlg\ge \theta^{-1}$, $\Delta_{lg}$ is essentially self-adjoint on the domain $C_{lg}^2(\Omega_\Phi)$, and therefore the form $Q_{lg}$ is closable.
For $\rlg<\theta^{-1}$ the form is not closable.

\paragraph{The transversal form}
We show that $q_{tr,n}^v(f,g)$ decomposes, for suitable $f,g$, into two pieces: one which has a limit, and the other which is oscillating with a phase which will be assumed to be non resonant (Definition~\ref{def-resphase}).
This will prove that \([D_{tr},f_{tr}]^\ast [D_{tr},g_{tr}]\) is the sum of a strongly regular and a non resonant part, and will imply by Corollary~\ref{cor-prodstrongreg} and Lemma~\ref{lem-stateres} the decomposition of $Q_{tr}$ claimed in equation~\eqref{eq-DirForm-tr2}.

We wish to adapt the parameter $\rtr$ so as to obtain a non-trivial form with core which is dense in $C(\Xi_{t_v})$.

One might be tempted to consider functions which are transversally locally constant to define its core.
However, it quickly becomes clear that $Q_{tr}$ vanishes on such functions. 
Indeed, if $f\in C(\Omega_\Phi)$ is transversally locally constant, then there exists $n_f\in \NM$, such that $\Rr^{-1}(\Pi_{0,n_f}(\Gg))$ gives a partition of $\Xi$ on which $f$ is constant.
So we see from equation~\eqref{eq-deltaef} that $\dtr_{e}f = 0$ for all $e\in E_{tr,n}$, $n\ge n_f$, and hence $Q_{tr}(f,f)=0$.

We will therefore consider a different core, namely the space generated by the eigenfunctions of the action.
We assume that these form a dense set of $L^2(\Omega_\Phi)$ which is equivalent to the fact that the tiling is pure point diffractive.

Consider an eigenfunction to eigenvalue $\beta$, $f_\beta:\Omega_\Phi\to\CM$.
By definition $f_\beta$ satisfies \(f_\beta(\omega+r)=e^{2\pi \beta(r)} f_\beta(\omega)\) for all $\omega\in \Omega_\Phi$ and $r\in \RM^d$.
Given $v\in\Vv$, $h\in \hat \Hh_{tr}$, and any $e\in E_{tr,n}^{v}(h)$, we have \(\Rob^{-1}(r(e)) = \Rob^{-1}(s(e)) + r_e= \Rob^{-1}(r(e)) + \theta^n r_h\) by equation~\eqref{eq-transtr}.
Then equation~\eqref{eq-deltaef} gives
\[ 
\dtr_{e}f_\beta = (e^{2\pi \imath \theta^n \beta(h)} - 1) f_\beta\circ\Rob^{-1}(s(e)) 
\]
and hence
\begin{multline*}
q_{tr,n}^v(f_\beta,f_{\beta'}) =
\frac{1}{\# E_{tr,n}^v}
\sum_{h\in \hat \Hh_{tr}}
\frac{\bigl(e^{-2\pi \imath\beta(h)\theta^n} - 1\bigr)\bigl( e^{2\pi \imath\beta'(h)\theta^n} - 1\bigr)}{\rtr^{2n}} \\
 \sum_{e \in E_{tr,n}^v(h)}
\overline{f_\beta\circ\Rob^{-1}(s(e))} f_{\beta'}\circ\Rob^{-1}(s(e))\,.
\end{multline*}
As for the longitudinal form, we average over boxes which partition $\Xi_{t_v}$: \(B_\gamma= \Rob^{-1}([\gamma])\), for \(\gamma \in \Pi_{0,l}(\Gg)\), and $1<<l<<n$, and we get
\[
q_{tr,n}^v(f_\beta,f_{\beta'})  \simeq 
c_{tr} \sum_{ h\in \hat \Hh_{tr}}
\frac{\bigl(e^{-2\pi \imath\beta(h)\theta^n} - 1\bigr)\bigl( e^{2\pi \imath\beta'(h)\theta^n} - 1\bigr)}{\rtr^{2n}}
\freq(t_{s^2(h)})
\int_{\Xi_{t_v}} \overline{f_\beta} f_{\beta'} d\mu_{tr}^v
\]
with
\begin{equation}\label{eq-ctr}
c_{tr}= \Bigl( \sum_{h\in\hat\Hh_{tr}}L_{s^2(h)}\Bigr)^{-1}.
\end{equation}
In view of equation~\eqref{eq-DirForm-tr2} let us consider
\[
q_{tr,n}(f,g) = \sum_{v\in\Vv} n_v \; q_{tr,n}^v (f_{tr}^v,g_{tr}^v) \int_{t_v} \overline{f_{lg}^v} g_{lg}^v d\mu_{lg}^v \,.
\]
Summing those terms up over $v\in\Vv$, the integrals yield \(\int_{\Omega_\Phi} \bar{f}_\beta f_{\beta'} d\mu = \delta_{\beta \beta'} \|f_\beta\|^2\) by orthogonality of the set of eigenfunctions.
Hence we get
\[
q_{tr,n}(f_\beta,f_{\beta'}) \simeq  c_{tr} \delta_{\beta\beta'} \|f_{\beta}\|^2
\sum_{ h\in \hat\Hh_{tr}} \freq(t_{s^2(h)}) 
\left| \frac{e^{2\pi \imath\beta(h)\theta^n} - 1}{\rtr^{n}} \right|^2 \,.
\]
We now use the results of Section~\ref{ssect-Pisot}: 
Using equation~\eqref{eq-Pisot}, expanding the exponential,
and neglecting terms proportional to $ \left(\frac{|\theta_j|}{\rtr}\right)^{2n} $ against $ \left(\frac{|\theta_2|}{\rtr}\right)^{2n}$
if $|\theta_j|<|\theta_2|$
we approximate 
 \[\left| \frac{e^{2\pi \imath\beta(h)\theta^n} - 1}{\rtr^{n}} \right|^2 \simeq \left|\sum_{j=2}^J \frac{p_{\beta(h)}(\theta_j)\theta_j^n}{\rtr^{n}}
\right|^2 \simeq \left(\frac{|\theta_2|}{\rtr}\right)^{2n}  \sum_{j,j'=2}^L \overline{p_{\beta(h)}(\theta_j)} p_{\beta(h)}(\theta_{j'}) 
e^{\imath (\alpha_j-\alpha_{j'})n}\,.
\]
Here $L$ is such that $|\theta_j|=|\theta_2|$ for $2\le j \le L$ and we wrote $\theta_j = |\theta_2| e^{\imath\alpha_j}$. The second approximation is justified by the fact that $p_{\beta(h)}(\theta_{j})\neq 0$ for all $j$.
We assume now that the phases are non resonant (Definition~\ref{def-resphase}):
\[
 \alpha_j - \alpha_{j'} + 2\pi k + 2\pi \frac{\log\rtr}{\log\rlg} k' \neq 0\,, \quad \forall k,k' \in \ZM
\]
for all $j\neq j'$.
This implies that the oscillating parts of \([D_{tr},f_{tr}]^\ast [D_{tr},g_{tr}]\) are non resonant, and hence by Lemma~\ref{lem-resphi} do not contribute to $\Tt\bigl([D_{tr},f_{tr}]^\ast [D_{tr},g_{tr}] \otimes \id\bigr)$ in equation~\eqref{eq-DirForm-tr}.
Concerning the non oscillating part, so when $j=j'$, we clearly see that they vanish in the limit $n\to +\infty$ if $\rtr > |\theta_2|$ and that they  tend to $+\infty$, as $n\rightarrow +\infty$, if
$\rtr < |\theta_2|$.
Finally, 
if $ \rtr=|\theta_2|$ then
the non oscillating part of \(q_{tr,n}(f_\beta,f_{\beta})\) converges to
\[
- c_{tr} (2\pi)^2  \|f_{\beta}\|^2 \sum_{ h\in \hat\Hh_{tr}}\freq(t_{s^2(h)})\sum_{j=2}^L |p_{\beta(h)}(\theta_j)|^2 \,.
\]
Define the operator $\Delta_{tr}$ on the linear space of 
dynamical eigenfunctions by
\begin{equation}
\label{eq-trLaplace}
\Delta_{tr} f_\beta = - c_{tr} (2\pi)^2 \sum_{ h\in\hat \Hh_{tr}}\freq(t_{s^2(h)})\sum_{j=2}^L |p_{\beta(h)}(\theta_j)|^2 \ f_\beta \,.
\end{equation}
Then, on the space   of dynamical eigenfunctions the transversal form is given by 
\[
Q_{tr}(f_\beta,f_\beta)
 = \left\{ \begin{array}{ll}
\langle f_\beta,\Delta_{tr} f_\beta\rangle_{L^2(\Omega_\Phi,d\mu)}
 & \mbox{if } \rtr = |\theta_2| \\ 
0 & \mbox{if } \rtr > |\theta_2| \\
+\infty & \mbox{if } \rtr < |\theta_2|
\end{array}\right. \,.
\]
Clearly, $Q_{tr}$ is closable but trivial if $\rtr>|\theta_2|$, whereas   $Q_{tr}$ is not closable if
$\rtr<|\theta_2|$.

\paragraph{Main result}

We summarize here our results about the Dirichlet forms. For a Pisot number $\theta$ of degree $J>1$, we denote $\theta_j, j=2,\cdots J$, the other Galois conjugates in decreasing order of modulus. We write the sub-leading conjugates in the form
$\theta_j=|\theta_2|e^{\imath \alpha_j}$, $2\leq j\leq L$,  where $\alpha_j\in [0,2\pi)$. 
In particular, $|\theta_j|<|\theta_2|$ for $j>L$.
\begin{theo} 
\label{thm-DirForm}
Consider a Pisot-substitution tiling of $\RM^d$ with Pisot number $\theta$ of degree $J>1$.
Assume that for all $j\neq j'\leq L$ one has
\begin{equation}
 \label{eq-phasePisot}
 \alpha_j - \alpha_{j'} + 2\pi k + 2\pi \frac{\log|\theta_2|}{\log\theta} k' \neq 0\,, \quad \forall k,k' \in \ZM \,.
\end{equation}
Set $\rho_{lg} = \theta^{-1}$ and $\rho_{tr}=|\theta_2|$.

If the dynamical spectrum is purely discrete 
then the set of finite linear combinations of dynamical eigenfunctions is a core for $Q$ on which it is closable.
Furthermore, $Q = Q_{tr} + Q_{lg}$, and $Q_{tr/lg}$ has generator $\Delta_{tr/lg}=\sum_{h\in \hat\Hh_{tr/lg}}\Delta_{tr/lg}^h$ given by
\begin{eqnarray*}
\Delta_{lg}^h f_\beta &=&  -c_{lg}(2\pi)^2 \mbox{\em freq}(t_{s^2(h)})  \beta(a_h)^2 f_\beta ,\\
\Delta_{tr}^h f_\beta &=&  -c_{tr}(2\pi)^2 \mbox{\em freq}(t_{s^2(h)})  \langle \tilde{r_h}^\star,\beta\rangle^2 f_\beta 
\end{eqnarray*}
where 
\(c_{lg}= \bigl( \sum_{h\in\hat\Hh_{lg}}R_{s^2(h)} \bigr)^{-1}\) and
\(c_{tr}=\bigl( \sum_{h\in\hat\Hh_{tr}}L_{s^2(h)}\bigr)^{-1}\). 
\end{theo}
\begin{proof}
We have calculated the explicit formulas for $\Delta_{lg}^h$ and $\Delta_{tr}^h$ above. It remains to insert (\ref{eq-geom}) into (\ref{eq-trLaplace}).
Eigenfunctions to distinct eigenvalues are orthogonal and hence form an orthogonal basis for the Hilbert space by assumption.
Standard arguments show then that $\Delta^h_{tr/lg}$ is essentially self-adjoint and hence the form closable.
\end{proof}

Notice that the ratio $\log(\theta)/\log(|\theta_2|)$ in
equation~\eqref{eq-phasePisot} is irrational unless $\theta$ is a
Pisot number of degree $J=3$ and unimodular \cite{Wald}.

\subsection{Geometric interpretation of the Laplacians}
We provide an interpretation of the Laplacians as
differential operators on the maximal equicontinuous factor $\hat E$
of the tiling system.  

Since the dynamical spectrum is pure point and all eigenfunctions
continuous the factor map $\pi:\Omega_\Phi\to\hat E $ 
becomes an isomorphism between
$L^2(\Omega,\mu)$ and $L^2(\hat E,\eta)$, where $\eta$ is the
normalized Haar measure on $\hat E$. The Dirichlet form  
$Q$ can therefore also be regarded as a form on $L^2(\hat E,\eta)$. 
We consider again first the simpler unimodular case in which 
$\hat E$ is a $dJ$-torus.
Each point in $\hat E$ has a tangent space which we may identify with
$U\oplus S$, the direct sum of the spaces tangent to the unstable and
the stable direction of $\varphi^t_\RM$, resp. 
Now the directional derivative at $x$ along $u\in U\oplus S$ is given by 
$$ (\langle u,\nabla\rangle f_\beta)(x)= \frac{d}{dt} f_\beta(x+t u)\left|_{t=0}
\right. = 2\pi i \langle u,\beta\rangle f_\beta(x).$$ 
We thus have
\begin{eqnarray*}
\Delta_{lg}^h &=&  c_{lg}\freq(t_{s^2(h)})  \langle\tilde a_h,\nabla\rangle^2  ,\\
\Delta_{tr}^h  &=&  c_{tr}\freq(t_{s^2(h)})  \langle\tilde{r_h}^*,\nabla\rangle^2 . 
\end{eqnarray*}
In the non-unimodular case one obtains essentially the same:  $\hat E$
is an inverse limit of tori $\hat F\cong \Ue/F^{rec}$ w.r.t.\ the 
map $\varphi^t_\RM$, 
$\hat E =\lim_\leftarrow (\Ue/F^{rec},\varphi^t_\RM)$. 
An element of $\hat E$ is a sequence $(x_n)_n$ of elements $x_n\in \Ue/F^{rec}$
satifying $x_n = \varphi^t_\RM(x_{n+1})$. The 
action $\alpha_r$ of $r\in\R^d$ on $\hat E$ is thus $\alpha_r((x_n)_n)
= (x_n+{\varphi^t_\RM}^{-n}(r))_n$.
The continuous
functions on $\hat E$ are the direct limit of continuous functions on
the tori w.r.t.\ the pull back of $\varphi^t_\RM$, 
$C(\hat E) =\lim_\rightarrow (C(\Ue/F^{rec}),{\varphi^t_\RM}^*)$.
The elements of $C^\infty (\hat E)$ are approximated by sequences $(f_n)_n$
of elements $f_n\in C^\infty(\Ue/F^{rec})$ which are
eventually $0$ modulo the equivalence relation identifying
$(0,\cdots,0,f,-{\varphi^t_\RM}^*f,0\cdots )$ with $0$. It follows
that the definition of the directional derivative is compatible with
the above equivalence relation so that the expression
$\langle u,\nabla \rangle(f_n)_n$ makes sense. In particular, since each
eigenfunction corresponds to an element of the form
$(0,\cdots,0,f_\beta,0\cdots )$, $\beta\in F$, we obtain the same
formula as above:  $ (\langle u,\nabla\rangle f_\beta)(x)=  2\pi i \langle u,\beta\rangle
f_\beta(x)$ and hence also the same formulae for $\Delta_{lg}^h$ and
$\Delta_{tr}^h$.

This shows that
the generators of the Dirichlet forms can be expressed as elliptic
second order differential operators on $\hat E$ with constant
coefficients. 


\end{document}